\documentclass[onefignum,onetabnum]{siamart190516}
\usepackage[mathscr]{euscript}
\usepackage{mathrsfs}
\usepackage{bbm}
\def\1{\mathbbm{1}}
\usepackage{amsfonts}
\usepackage{amssymb}
\usepackage{latexsym}
\newcommand{\R}{{\mathbb R}}
\newcommand{\T}{{\mathbb T}}                   
\newcommand{\N}{{\mathbb N}}                  
\newcommand{\C}{{\mathbb C}}                  
\newcommand{\F}{{\mathbb F}}

\newcommand{\A}{{\mathbb A}}

\def\al{\alpha}
\def\om{\omega}
\def\Om{\Omega}
\def\ga{\gamma}
\def\si{\sigma}

\def\la{\lambda}

\def\t{\tau}
\def\La{\Lambda}

\def\calQ{{\mathcal{Q}}}
\def\calM{{\mathcal{M}}}

\def\calL{{\mathcal{L}}}

\def\calT{{\mathcal{T}}}
\def\calS{{\mathcal{S}}}

\def\calG{{\mathcal{G}}}
\def\calA{{\mathcal{A}}}

\def\calC{{\mathcal{C}}}

\def\calX{{\mathcal{X}}}

\def\calK{{\mathcal{K}}}

\def\calF{{\mathcal{F}}}

\def\calH{{\mathcal{H}}}

\def\R{\mathbb R}

\def\P{\mathbb P}

\def\S{\mathbb S}
\def\C{\mathbb C}

\def\N{\mathbb N}

\def\E{\mathbb E}

\def\F{\mathbb F}
\def\K{\mathbb K}
\def\L{\mathbb L}

\def\T{\mathbb T}
\def\X{\mathbb X}

\usepackage{lipsum}
\usepackage{amsfonts}
\usepackage{graphicx}
\usepackage{epstopdf}
\usepackage{algorithmic}
\ifpdf
  \DeclareGraphicsExtensions{.eps,.pdf,.png,.jpg}
\else
  \DeclareGraphicsExtensions{.eps}
\fi

\newsiamremark{remark}{Remark}
\newsiamremark{example}{Example}
\newsiamremark{hypothesis}{Hypothesis}
\crefname{hypothesis}{Hypothesis}{Hypotheses}
\newsiamthm{claim}{Claim}

\headers{Infinite dimensional stochastic linear systems}{F.Z. Lahbiri and S. Hadd}

\title{A functional analytic approach to infinite dimensional stochastic linear systems
}

\author{Fatima Zahra Lahbiri and Said Hadd\thanks{Department of Mathematics, Faculty of Sciences, Ibn Zohr University, Hay Dakhla, BP8106, 80000--Agadir, Morocco
  (\email{fatimazahra.lahbiri@gmail.com, s.hadd@uiz.ac.ma}).}
}

\usepackage{amsopn}

\ifpdf
\hypersetup{
  pdftitle={A functional analytic approach to infinite dimensional stochastic linear systems},
  pdfauthor={F.Z. Lahbiri and S. Hadd}
}
\fi

\begin{document}

\maketitle

\begin{abstract}
In this paper, we study infinite dimensional stochastic systems having both unbounded control and observation operators. First of all, using a semigroup approach,  we give another take of the well-posedness of such systems treated in [SIAM J. Control Optim., 53 (2015), pp. 3457--3482]. Second, we propose a new variation of constants formula for mild solutions of perturbed abstract stochastic Cauchy problems using the concept of Yosida extensions of admissible operators. Third, we prove the well-posedness of perturbed boundary control systems. Fourth, we apply this result to a general class of stochastic systems with delays in state, control and observation parts. Finally, we study admissible observation operators and exact observability for semilinear stochastic systems.
\end{abstract}



\section{Introduction}\label{sec:intro}
In this paper we are concerned with the study of abstract stochastic systems having unbounded control and observation operators. It seems that the work \cite{flandoli1} is one of the first papers dealing with the existence and regularity of mild solutions of boundary control stochastic systems. A generalization of the concept of well-posed linear system to stochastic systems has recently been obtained in \cite{Lu-SIAM-15}. More recently, the well-posedness of boundary controlled and observed stochastic Port-Hamiltonian systems has been investigated in  \cite{LW-19} using an approach somewhat similar to that of \cite{Lu-SIAM-15}. All of these references consider a distributed noise. Another way to control a system is to consider a noise at the boundary conditions. These types of systems have been considered in e.g. \cite{Al-Bo-02}, \cite{Da-Za-SSR}, \cite{Fa-GO-09}. For works dealing with controllability and observability of infinite dimensional stochastic systems using Carleman estimate we refer to \cite{Lu-JDE-13}, \cite{Lu-SIAM-13} and \cite{Lu-Zhang-21}. In terms of facts, stochastic systems with bounded control and observation operators are well investigated in the literature, see e.g. \cite{cerrai0}, \cite{cerrai1}, \cite{Duncan-12}, \cite{Duncan-13}, \cite{Nazim2001} and the references therein. Particularly, in \cite{Duncan-12}, \cite{Duncan-13} the noise is governed by fractional Brownian motions. We note that the problem of ergodic control was recently studied in \cite{Alb}. We also refer to \cite{Alb1}, \cite{Alb2}, \cite{Alb3} for  a study of infinite dimensional stochastic systems using the concept of  invariant measures.

In this paper, we combine stochastic analysis with the well-established class of deterministic regular linear systems to study the well-posedness and observability of stochastic systems with unbounded control and observation operators. Using a functional analytic approach, we first propose another take to the approach introduced in \cite{Lu-SIAM-15}. Furthermore, we propose a new variation of constants formula for the mild solutions of perturbed stochastic abstract Cauchy problems using the concept of Yosida extensions of admissible observation operators. The obtained results are applied to the well-posedness of stochastic linear systems with distributed delays in state, control and observation parts. We also study admissible observation operators for semilinear stochastic systems. In order to clearly state our results, we introduce some notations. Hereafter $H,U$ and $\mathscr{Y}$ are Hilbert spaces, $(\Omega,\calF,\F,\P)$ is a filtered probability space with a natural filtration $\F=(\calF_t)_{t\ge 0}$ which is generated by the standard one-dimensional Brownian motion $(W(t))_{t\ge 0}$. For $t<0$, $\calF_t$ is taken to be $\calF_{0}.$  We consider $A:D(A)\subset H\to H$ is a generator of a $C_0$-semigroup $\T:=(T(t))_{t\ge 0}$ on  $H,$  and denote by $H_{-1}$ is the completion of $H$ with respect to the norm $\|x\|=\|R(\la,A)x\|$ for some (hence all) $\la\in\rho(A)$ and $x\in H$.

We are concerned with the well-posedness and control properties of the following abstract input/output stochastic system
\begin{align}\label{ABC-uy}
\begin{cases}
dX(t)=(AX(t)+Bu(t))dt+\mathscr{M}(X(t))dW(t),& t>0,\quad X(0)=\xi,\cr Y(t)=CX(t),& t> 0.
\end{cases}
\end{align}
where,  $B\in\calL(U,H_{-1})$ is a control operator, $C\in\calL(D(A),\mathscr{Y})$ is an observation operator and $\mathscr{M}\in\calL(H)$. The initial process $\xi\in L^2_{\calF_0}(\Om,H)$ and the control function $u\in L^2_{\F}(0,+\infty;U)$ (see the definition of this space below).

In the absence of the noise (i.e. the deterministic case), the system \eqref{ABC-uy} is well studied, see e.g. \cite{Cu-Zw}, \cite{Sala}, \cite{Staf}, \cite{TucWei}, \cite{WeiRegu} and the references therein. The unboundedness of the operators $B$ and $C$ leads to the notion of admissibility in order to have continuous dependence of $L^2$-controls or continuous dependance of $L^2$-observations on states, see \cite{TucWei}. For $L^2$-well-posed linear systems one has in addition continuous dependance of $L^2$-observation on $L^2$-controls, see \cite{Staf}. On the other hand, Weiss \cite{WeiRegu} introduced an importance subclass of infinite-dimensional well-posed linear systems called regular linear systems for which one can study feedback theory and feedback stabilization.

In the presence of the noise, the well-posedness of the system \eqref{ABC-uy} was recently studied in \cite{Lu-SIAM-15}, where the main objective was to extend the concept of well-posed linear system to the stochastic system \eqref{ABC-uy}. In fact, the author defined a concept of admissibility for the unbounded operators $B$ and $C$ similar to the deterministic systems. He also proved that the mild solution to the system   \eqref{ABC-uy} is also a weak solution to \eqref{ABC-uy}, which is important if one want to apply It\^{o} formula.

In order to summarize the results of the present paper, we introduce the following notation. For any sub-$\si$-algebra $\mathscr{N}$ of $\calF$, denote by  $L^{2}_{\mathscr{N}}(\Om,H)$ the set of all $\mathscr{N}$--measurable ($H$-valued) random variables $\xi:\Om\to H$ such that $\E\|\xi\|^2<\infty$. On the other hand, if $\mathscr{X}$ is a Hilbert space and $\t\in(0,\infty]$ a real number, we denote
\begin{align*}
L^2_{\F}(0,\t;\mathscr{X}) :=\left\{\zeta:[0,\t]\times \Om\to \mathscr{X}:\zeta \;\text{is}
\;\F-\text{adapted and}\; \int^\t_0 \mathbb{E}\|\zeta(t)\|_{\mathscr{X}}^2<\infty\right\}
\end{align*}
and the space of  mild solutions for stochastic equations
\begin{align*}
& \mathcal{C}_{\F}\left(0,\t;L^2(\Om,\mathscr{X})\right)\cr & \quad :=\left\{\zeta:[0,\t]\times \Om\to \mathscr{X}:\zeta \;\text{is}\;\F-\text{adapted and}\; t\mapsto\left(\mathbb{E}\|\zeta(t)\|_{\mathscr{X}}^2\right)^{\frac{1}{2}}\;\text{is continuous}\right\}.
\end{align*}
For any $\al>0$ and $\zeta\in L^2_{\F}(0,\al;H),$ we denote
\begin{align*}
(\T\diamond \zeta)(t):=\int^t_0T(t-s)\zeta(s)dW(s),\qquad t\in [0,\al].
\end{align*}

In Section \ref{Section2}, on the one hand, we recall the concept of regular linear systems in Weiss sense \cite{WeiRegu}. On the other hand, by assuming that $B$ is an admissible control operator in deterministic sense and  reformulating the stochastic differential equation in \eqref{ABC-uy} to an abstract stochastic Cauchy problem on a product space, we given a new proof to the existence of a unique mild solution $X(\cdot)\in \mathcal{C}_{\F}\left(0,+\infty;L^2(\Om,H)\right)$ to the system \eqref{ABC-uy}, see Theorem \ref{thm-existence}. Furthermore, if we assume that $(A,B,C)$ is a regular system in the deterministic sense and if we denote by $C_\Lambda$ the Yosida extension of $C$ for $A$, then the mild solution $X(\cdot)$ of the stochastic system \eqref{ABC-uy} satisfies $X(t)\in D(C_{\Lambda})$ for a.e. $t>0$, $\P$-a.s, and
\begin{align*}
\|C_\Lambda X(\cdot)\|_{L^2_{\F}(0,\al;\mathscr{Y})}\le \kappa \left(\|\xi\|_{L^{2}_{\calF_0}(\Om,H)}+ \|u\|_{L^2_{\F}(0,\al;U)}\right)
\end{align*}
for any $\al>0,$ $\xi\in L^{2}_{\calF_0}(\Om,H)$ and $u\in L^2_{\F}(0,\al;U)$, where $\kappa:=\kappa(\al)>0$ is a constant, see Theorem \ref{ABCW-wellposed}. The key of the proof of this theorem is the following results on the stochastic convolution: if $C$ is an admissible observation operator for $A$ (in the deterministic sense) and $\zeta\in L^2_{\F}(0,\al;H),$ then  $(\T\diamond \zeta)(t)\in D(C_\Lambda)$ for a.e. $t\ge 0,$ $\P$-a.s., and
\begin{align}\label{C-Tcv}
\mathbb{E}\int_{0}^{\al}\|C_{\Lambda}(\T\diamond \zeta)(t)\|^{2}_{\mathscr{Y}} dt\leq \ga^2\mathbb{E}\int_{0}^{\al}\|\zeta(s)\|^{2}_{H}ds
\end{align}
for some constant $\ga:=\ga(\al)>0$ independent of $\zeta$, see Proposition \ref{fondamental-lemma}. We mention that we cannot permute the sign integral with the operator $C$ (or $C_\Lambda$) because in systems theory $C$ is neither closed nor closeable.

In Section \ref{sec:new-VCF}, we first prove that for an admissible observation operator $\mathscr{P}\in\calL(D(A),H)$ for $A$, the following perturbed stochastic Cauchy problem
 \begin{align}\label{perturbed-S-Cauchy}
 dX(t)=(A+\mathscr{P})X(t)dt+\mathscr{M}(X(t))dW(t),\quad t>0,\quad X(0)=\xi,
 \end{align}
 has a unique mild solution if $\mathscr{P}\in\calL(D(A),H)$ such $X(t)\in D(\mathscr{P}_{\Lambda})$ for a.e. $t>0,$ $\P$-a.s., and
 \begin{align*}
 X(t)=T(t)\xi+\int^t_0 T(t-s)\mathscr{P}_{\Lambda}X(s)ds+\int^t_0 T(t-s)\mathscr{M}(X(t))dW(s)
 \end{align*}
 for any $t\ge 0$ and $\xi\in L^2_{\calF_0}(\Om,H)$, where $\mathscr{P}_{\Lambda}$ is the Yosida extension of $\mathscr{P}$ for $A$, see Theorem \ref{new variation}.  With the help of this formula we prove the well-posedness of the perturbed boundary control stochastic problem defined by \eqref{P-ABC}. This result extend the result in \cite{Lu-SIAM-15} because the latter uses bounded perturbations of the generator.

In Section \ref{sec:delay}, we prove that a general class of delay systems (see the system \eqref{SSdelay}) is well-posed. This extends the results obtained in \cite{Ha-Id-IMA} which deals with deterministic delay systems.

In Section \ref{sec:semilinear}, using the inequality \eqref{C-Tcv}, we prove the well-posedness and exact observability of the semilinear stochastic system
 \begin{align}\label{semilinear-S}
\begin{cases}
dX(t)=(A X(t)+G(X(t)))dt+F(X(t))dW(t),\quad X(0)=\xi,& t\ge 0,\cr  Y(t)=CX(t),& t\ge 0.
\end{cases}
\end{align}
where $G$ and $F$ are nonlinear applications from $H$ to $H$ satisfying suitable (globally) Lipschitz conditions. These results extend those obtained in \cite{BJ-09} for semilinear deterministic systems.


\section{Well-posedness of input-output stochastic linear systems}\label{Section2}
The object of this section is to introduce a semigroup approach to the well-posedness of the  stochastic linear system \eqref{ABC-uy}. Before doing so, we first give a concise background on the well established theory of regular linear systems in the Weiss sense \cite{WeiRegu}, \cite{WeiTrans}.
In the sequel we use the notation introduced in Section \ref{sec:intro} related to the system \eqref{ABC-uy}. We also denote by $\T_{-1}=(T_{-1}(t))_{t\ge 0}$ the $C_0$-semigroup on $H_{-1}$, extension of the semigroup $\T$ to $H_{-1}$. The generator of $\T_{-1}$ is the extension of $A$ to $H,$ which will be denoted by $A_{-1}:H\to H_{-1}$, see e.g. \cite[Chap.II]{EngNag} for more details on extrapolation theory. Moreover, for any operator $M\in\calL(D(A),\X)$ ($\X$ is a Hilbert space), we consider its {\em Yosida extension} for $A$, defined by
\begin{align}\label{Yosida-extension-def}
\begin{split}
D(M_\Lambda)&:=\left\{x\in H: \lim_{\la\to +\infty}  M\la R(\la,A)x\;\text{exists in}\;\X\right\},\cr M_\Lambda x&:= \lim_{\la\to +\infty}  M\la R(\la,A)x.
\end{split}
\end{align}
This operator depends on $A,$ so $M_\Lambda$ may change relative to another generator, see \cite{Ha-Id-SCL}. As we will see below a such operator will play an important role in the representation of output functions of infinite-dimensional linear systems.

\subsection{Regular linear systems}\label{sec:RLS}
An operator $B\in\calL(U,X_{-1})$ is called an {\em admissible control operator} for $A,$ if there exists a real $\t>0$ such that
\begin{align}\label{control-maps}
\Phi_\t u:=\int^\t_0 T_{-1}(t-s)Bu(s)ds\in H
\end{align}
for any  $u\in L^2([0,+\infty),U)$. By the closed graph theorem, the admissibility of $B$ for $A$ implies that $\Phi_t:L^2([0,+\infty),U)\to H$ is a linear bounded for any $t\ge 0$, see \cite[Chap. 4]{TucWei}.

We say that the system $(A,B)$ is $\t$-exactly controllable ($\t>0$) is $B$ is an admissible control operator for $A$ and ${\rm range}(\Phi_\t)=H$.

An operator $C\in\calL(D(A),\mathscr{Y})$ is called an {\em admissible observation operator} for $A,$ if for some (hence all) $\al>0,$ there exists a constant $\ga:=\ga(\al)>0$ such that
\begin{align}\label{Obser-Estim0}
\int^{\al}_0 \|CT(t)x\|^2_{\mathscr{Y}} dt\le \ga^2 \|x\|^2,\quad \forall x\in D(A).
\end{align}
The admissibility of $C$ for $A$ implies that the map $\Psi:D(A)\to L^2([0,\al],\mathscr{Y})$ defined by $(\Psi x)(t)=CT(t)x$ has a bounded extension to $H$. This map is called the {\em extended output map} associated with $A$ and $C$. In this case, it is shown in \cite[Chap. 4]{TucWei} that for any $x\in H,$ $T(t)x\in D(C_\Lambda)$ for a.e. $t>0$, and $\Psi x=C_\Lambda T(\cdot)x$ a.e. on $(0,+\infty)$. Note that if $C$ is admissible for $A$, then we can replace $C$ by $C_\Lambda$ and the inequality \eqref{Obser-Estim0} holds for any $x\in H$.

The proof of the following result follows immediately from \cite[Proposition 3.3]{Hadd-SF}.
\begin{proposition}\label{Hadd-SF-prop}
Assume that $C\in\calL(D(A),\mathscr{Y})$ is an admissible observation operator for $A$. For any $\al>0$ there exists a constant $\ga:=\ga(\al)>0$ such that for  any $\zeta\in L^2_{\F}(0,\al;H)$,
\begin{align*}&\int_0^t T(t-s)\zeta(s)ds\in D(C_{\Lambda}),\quad for\;a.e.\; t\geq0,\; \P-a.s.\; and,\cr & \E\int^\al_0\left\|C_{\Lambda}\int_0^t T(t-s)\zeta(s)ds\right\|_{\mathscr{Y}}^2dt\leq \al \ga^2\; \E\int^\al_0 \|\zeta(s)\|_H^2 ds.\end{align*}
\end{proposition}

Assume now that $B\in\calL(U,H_{-1})$, $C\in\calL(D(A),\mathscr{Y})$ and consider the deterministic input-output linear system
\begin{align*}
(A,B,C)\quad\begin{cases}
\dot{x}(t)=Ax(t)+Bu(t),& t>0,\quad x(0)=x_0,\cr y(t)=Cx(t),& t>0.
\end{cases}
\end{align*}
Remark that if $B$ is admissible for $A,$ then the system $(A,B,C)$ has a unique mild solution $x(\cdot)\in \calC([0,+\infty),H)$ such that $x(t)=T(t)x_0+\Phi_t u$ for any $t\ge 0$, $x_0\in H$ and $u\in L^2([0,+\infty),U)$. The main question is how to deal with the output function $t\mapsto y(t)$. In systems theory, we require $y(\cdot)$ to be an $L^2$-function in order to investigate the feedback theory, the stabilization and the linear quadratic problem of a linear system. As $C$ is only acting on $D(A)$ and $x(t)\in H$ for admissible control operators, then the expression "$CX(t)$" is not well-defined  even if the observation operator $C$ is admissible for $A$. The main reason for this obstacle is that the term "$C\Phi_t u$" is not well-defined. This leads to the concept of well-posed linear systems developed first by Salamon \cite{Sala} and latter by Staffans \cite{Staf}, and Weiss \cite{WeiRegu}. In order to define this notion, we denote $(\S_\t u)(t)=u(t+\t)$ for $t,\t\ge 0$ and $u\in L^2([0,+\infty),V)$ for $V=U,Y$ (for simplicity we use the same notation).
\begin{definition}\label{D-wp}
Let $B\in\calL(U,H_{-1})$, $C\in\calL(D(A),Y)$. The system $(A,B,C)$ is well-posed on $X,U,Y$ if $B$ is an admissible control operator for $A,$ $C$ is an admissible observation operator for $A,$ and there exists a linear bounded operator $\mathscr{F}:L^2([0,\al],U)\to L^2([0,\al],\mathscr{Y})$ for $\al>0$ with
\begin{align*}
\S_\t \mathscr{F} u =\Psi \Phi_\t u_{|[0,\t]}+\mathscr{F}\S_\t u,\quad \t>0,\quad u\in L^2_{loc}([0,+\infty),U),
\end{align*}
such that the output function is extended to a function of the form
\begin{align*}
y=\Psi x_0+\mathscr{F}u\quad a.e.\;\text{on}\;(0,+\infty)
\end{align*}
for any $x_0\in H$ and $u\in L^2_{loc}([0,+\infty),U)$.
\end{definition}
An important subclass of well-posed system is introduced in \cite{WeiRegu} and defined as:
\begin{definition}\label{Reg-syst}
A well-posed triple $(A,B,C)$ is called regular (with feedthrough zero) if the following limit
               \begin{align*}
               \lim_{\t\to 0^+}\frac{1}{\t}\int^{\t}_0 \left( \mathscr{F}(\1_{\R^+}(\cdot) v)\right)(s)ds=0
               \end{align*}
               exists for a constant control $v\in U$.
\end{definition}
The following result gives a representation of the output functions of regular linear systems, where the proof can be found in \cite{WeiRegu}.
\begin{theorem}\label{Weiss-repres}
Assume that the triple $(A,B,C)$ is regular. Then $\Phi_t u\in D(C_\Lambda)$  and $(\mathscr{F}u)(t)=C_\Lambda \Phi_t u$
for almost every $t\ge 0,$ and $u\in L^2_{loc}([0,+\infty),U)$. In particular,
the state and the output functions of the system $(A,B,C)$
satisfy $x(t;x_0,u)\in D(C_\Lambda)$ and $y(t;x_0,u)=C_\Lambda x(t)$ for any $x_0\in H,$ $u\in L^2_{loc}([0,+\infty),U)$ and
almost every $t> 0$.
\end{theorem}
\begin{remark}\label{nice-rem}
According to \cite{WeiTrans}, in the Hilbert setting, the regularity of the triple $(A,B,C)$ is  equivalent to ${\rm Range} (R(\la,A_{-1})B)\subset D(C_\Lambda)$ for some (hence all) $\la\in\rho(A)$. In this case, the transfer function of the system $(A,B,C)$ is defined by $\mathbb{G}(\la)=C_{\La}R(\la,A_{-1})B$ for $\la\in\rho(A),$ and satisfies $\mathbb{G}(\la)\to 0$ as $\la\to+\infty$.
\end{remark}

\subsection{A semigroup approach to the well-posed stochastic linear systems}\label{Stoch-WP-S}
The well-posedness of the stochastic system \eqref{ABC-uy} is already considered in \cite{Lu-SIAM-15}. Here in this article, we give another take on well-posedness using a semigroup approach and Yosida extensions. We adopt the following definition (see also \cite[Definition 1.1]{Lu-SIAM-15}).
\begin{definition}
A process $X(\cdot)$ is said to be a mild solution of the stochastic linear system \eqref{ABC-uy} if
\begin{enumerate}
\item $X(t)\in H,$ $\P$ almost surely,
\item $X(\cdot)\in \calC_{\F}(0,+\infty;L^2(\Om,H))$,
  \item for any $t\ge 0$ and $\xi\in L^2_{\calF_0}(\Om,H)$,
\end{enumerate}
\begin{equation}\label{B-mild-sol}X(t)=T(t)\xi+\int^t_0 T_{-1}(t-s)Bu(s)s+\int^t_0 T(t-s)\mathscr{M}(X(s))dW(s).\end{equation}
\end{definition}
In the sequel, we need the following maps: for any $t\ge 0,$
\begin{align}\label{Operator-map}
\tilde{\Phi}_t: L^2_{\F}(0,+\infty;U)\to L^2_{\calF_t}(\Om,H_{-1}),\quad \tilde{\Phi}_{t}u:=\int_{0}^{t}T_{-1}(t-s){B}u(s,.)ds.
\end{align}
It is clear that if the operator $B$ is admissible for $A,$
\begin{align}\label{yaya}
\tilde{\Phi}_t\in \calL\left(L^2_{\F}(0,+\infty;U),L^2_{\calF_t}(\Om,H)\right),\quad \forall t\ge 0.
\end{align}

The first main result of this section is the following:
\begin{theorem}\label{thm-existence}
Assume that $B$ is an admissible control operator for $A$. For every $\xi\in L^{2}_{\mathcal{F}_{0}}(\Omega,H)$
and $u\in L^2_{\F}(0,+\infty;U)$, there exists a unique mild solution $X(\cdot;\xi,u)\in \calC_{\F}(0,+\infty;L^2(\Om,H))$ of the stochastic system \eqref{ABC-uy} satisfying
\begin{align}\label{reformulation}
X(t;\xi,u)=T(t)\xi+\int_{0}^{t}T(t-s)\mathscr{M}(X(s,\xi,0))dW(s)+\Phi^{W}_{t}u,\qquad t\ge 0,
\end{align}
where
\begin{align}\label{stoch-maps}
\begin{split}
&\Phi^{W}_{t}u=\tilde{\Phi}_{t}u+\int_{0}^{t}T(t-s)\mathscr{M}(\Phi^{W}_{s}u)dW(s),\quad and\cr
&\Phi^{W}_{t}\in \calL\left(L^2_{\F}(0,+\infty;U),L^{2}_{\mathcal{F}_{t}}(\Omega;H)\right),\qquad t\ge 0.
\end{split}
\end{align}
\end{theorem}
\begin{proof}
Let $t_0>0$ be arbitrary and define the following spaces
\begin{align*}
&\mathfrak{X}:=H\times L^2([0,+\infty),U).
\end{align*}
Consider the matrix operator
\begin{align}\label{lax-gen}
\begin{split}\mathfrak{A}&:=\begin{pmatrix}
A_{-1} & B\delta_0\\
0 & \frac{d}{ds}
\end{pmatrix},\cr D(\mathfrak{A})&:=\left\{\binom{x}{g}\in H\times W^{1,2}([0,+\infty),U):A_{-1}x+Bg(0)\in H \right\}.
\end{split}
\end{align}
Clearly, $(\mathfrak{A},D(\mathfrak{A}))$ generates a $C_0$--semigroup $\mathfrak{T}:=(\mathfrak{T}(t))_{t\ge 0}$ on $\mathfrak{X}$ such that
\begin{align}\label{lax-sem}\mathfrak{T}(t)=\begin{pmatrix}
T(t) & \Xi(t)\\
0 & S_U(t)
\end{pmatrix},\quad t\ge0,
\end{align}
where \begin{align*}\Xi(t)u:=\int_{0}^{t}T_{-1}(t-s)B\delta_{0}S_{U}(t)uds,\qquad t\ge 0. \end{align*}
Here $(S_U(t))_{t\ge0}$ is the right shift semigroup on $L^2([0,+\infty),U).$ Introducing the state
\begin{equation*}
  \rho(t)=\binom{X(t)}{S_U(t)u},\quad t\geq0.
\end{equation*}
The problem \eqref{ABC-uy} becomes
\begin{align}\label{abstract-forme}
\begin{cases}
		d{\varrho}(t)=\mathfrak{A}\rho(t)dt+\mathscr{B}(\rho(t))dW(t), & t>0, \\
		\varrho(0)=\binom{\xi}{u},
	\end{cases}
\end{align}
where $\mathscr{B}(\varrho(t))=\binom{\mathscr{M}(X(t))}{0}$. It is well-known (see e.g. \cite[Chap.5]{Da-Za-Book-96})
that the system \eqref{abstract-forme} has a unique mild solution $\varrho\in \calC_{\F}(0,+\infty;L^2(\Om,\mathfrak{X}))$.
Thus the first projection of $\varrho(t)$ satisfies $X(\cdot)\in \calC_{\F}(0,+\infty;L^2(\Om,H))$ and \eqref{B-mild-sol}.
This proves the first part of the theorem. Now select $\Phi^{W}_t u:=X(t,0,u)$ for $t\ge 0$ and
$u\in L^2_{\F}(0,+\infty;U)$. Thus by taking $\xi=0$ in the equation \eqref{B-mild-sol},
we obtain the first equation in \eqref{stoch-maps}. On the other hand, let $\t>0$ be arbitrary,
$\delta>\om_0(A)$ and  $M\ge 1$ such that $\|T(t)\|\le M e^{\delta t}$ for any $t\ge 0$. Let $t\in [0,\t]$ and $u\in L^2_{\F}(0,\t;U)$. Using \eqref{yaya} and It\^{o}'s isometry, we estimate
\begin{align*}
\mathbb{E}\|\Phi^{W}_{t}u\|^{2}_{H}
&\le 2 c \|u\|^{2}_{L^{2}_{\F}(0,\t;U)}+ 2 \mathbb{E}\left\|\int_{0}^{t}T(t-s)\mathscr{M}(\Phi^{W}_{s}u)dW(s)\right\|^{2}_{H}\\
&\leq 2 c \|u\|^{2}_{L^{2}_{\F}(0,\t;U)}+ 2 (M\|\mathscr{M}\|)^2 e^{2|\delta|\t} \int_{0}^{t}\E\|\Phi^{W}_{s}u\|^2 ds,
\end{align*}
for a constant $c:=c(\t)>0$. Now by applying  Gronwall's inequality and exploiting the continuity of trajectories,  we have
\begin{align}\label{phi1}\mathbb{E}\|\Phi^{W}_{\t}u\|^{2}_{H}\leq \tilde{\gamma}\, \|u\|^{2}_{L^{2}_{\F}(0,\t;U)}
,\end{align}
for a constant $\tilde{\gamma}:=\tilde{\gamma}(\t)>0$. Let us now prove the linearity of each $\Phi^{W}_{\t}$.
In fact, from the first part of the proof, for each $u\in L^2_{\F}(0,+\infty;U),$
the process $X^{u}(t)=\Phi^{W}_{t}u$ is the unique mild solution of the following stochastic equation
\begin{align}\label{phistocu-0}
d{X}^{u}(t)=\left(AX^{u}(t)+Bu(t)\right)dt+\mathscr{M}(X^{u}(t))d{W}(t),\quad X^{u}(0)=0,\quad  t\ge 0.
\end{align}
Let $u_1,u_2\in L^2_{\F}(0,+\infty;U),$ and $\la\in\C$. By uniqueness of the solution of \eqref{phistocu-0},
it suffices to show that $X^{u_1+\lambda u_2}= X^{u_1}+X^{\lambda u_2}$. This can also reduced to prove that $X^{\la u_2}=\la X^{u_2}$.
In fact, as each $\tilde{\Phi}_t$ is linear, by the expression of $\Phi^{W}_t$ in \eqref{stoch-maps}, It\^{o}'s Isometry and Fubini's Theorem, we obtain
\begin{align*}
  \mathbb{E}\|X^{\lambda u_2}(t)-\lambda X^{u_2}(t)\|^{2}_{H} &= \E\left\|\int^t_0 T(t-s)\mathscr{M}\left(X^{\la u_2}(s)-\la X^{ u_2}(s)\right)dW(s)\right\|\cr & \le (M\|\mathscr{M}\|)^2 e^{2|\delta|t} \int^t_0 \mathbb{E}\|X^{\lambda u_2}(s)-\lambda X^{u_2}(s)\|^{2}_{H} ds.
\end{align*}
Using Gronwall's inequality, we obtain
\begin{align*}
  X^{\lambda v}(t)=\lambda X^{v}(t),\quad t\geq0.
\end{align*}
Moreover, the formula in \eqref{reformulation} follows once we remark that
\begin{align*}X(t,\xi,u)=X(t,\xi,0)+X(t,0,u),\end{align*}
where \begin{align}\label{walid-hadd}
        X(t,\xi,0)=T(t)\xi+\int_{0}^{t}T(t-s)\mathscr{M}(X(s,\xi,0))dW(s),
      \end{align}
for $t\in[0,\t]$ and $\xi\in L^{2}_{\calF_0}(\Om,H)$. This ends the proof.
\end{proof}
\begin{remark}
The existence and uniqueness of the mild solution of the system \eqref{ABC-uy} is proved in \cite[Theorem 2.1]{Lu-SIAM-15}
using a direct approach and also a fixed point theorem. On the other hand, a detailed study of the existence and regularity of mild solution of the system \eqref{ABC-uy} in the parabolic case is given in \cite{flandoli1}. In the present paper, we used a completely different approach based on the transformation
of the control system \eqref{ABC-uy} into a standard stochastic Cauchy problem using the Lax-Phillips semigroup \eqref{lax-sem}.
Furthermore, we rewritten the mild solution $X(\cdot,\xi,u)$ as $X(t,\xi,u)=Z(t,\xi)+\Phi^{W}_t u$,
where  $Z(t,\xi)=X(t,\xi,0)$ is the unique solution of the stochastic Cauchy problem
\begin{align*}
dZ(t)=AZ(t)dt+\mathscr{M}(Z(t))dW(t),\qquad t>0,\quad Z(0)=\xi.
\end{align*}
\end{remark}
\begin{definition}
Assume that the control operator $B$ is admissible for $A$ and let $\t>0$. The system \eqref{ABC-uy} is called $\t$-exactly controllable (or exactly controllable on $[0,\t]$) if for any $\xi\in L^{2}_{\calF_0}(\Om,H)$ and $\zeta\in L^{2}_{\calF_\t}(\Om,H),$ there exists a control $u\in L^{2}_{\F}(0,\t;U)$ such that the corresponding solution $X(.)$ satisfies $X(\t)=\zeta$.
\end{definition}
\begin{remark}\label{exact-controlla-stoch}
According to the equation \eqref{stoch-maps},
to prove the $\t$-exact controllability of the system \eqref{ABC-uy} it suffices to show that
\begin{equation}\label{stoch-exact-condition}
  {\rm Range}(\Phi^W_{\tau})=L^{2}_{\calF_{\t}}(\Omega,H).
\end{equation}
In fact, if $\zeta\in L^{2}_{\calF_\t}(\Om,H),$ then we also have
\begin{align*}
\ell:=\zeta-T(\t)\xi-\int^{\t}_0 T(\t-s)\mathscr{M}(X(s,\xi,0))dW(s)\in L^{2}_{\calF_{\t}}(\Omega,H).
\end{align*}
Now if the condition \eqref{stoch-exact-condition} holds, then there exists $u\in L^2_{\F}(0,\t;U)$ such that
$
\Phi^{W}_{\tau}u=\ell.
$
This means that $X(\t;\xi,u)=\zeta$.
\end{remark}
\begin{remark}\label{result-exact-contro}
{\rm (i)} It is shown in \cite[Theorem 3.2]{Nazim2001} that a stochastic system with bounded control operator is exactly controllable on $[0,\t]$ if and only if the corresponding deterministic system is exactly controllable in $[s,\t]$ for any $s\in [0,\t)$. However, when the control operators are unbounded it is not clear how to extend the approach used in \cite{Nazim2001}. \\
 {\rm (ii)} Let the assumptions of Theorem \ref{thm-existence} be satisfied. Moreover, we assume that the map $\tilde{\Phi}_\t:L^2_{\F}(0,\t,U)\to L^2_{\mathcal{F}_\t}(\Omega,H)$ is surjective. Now  for $\zeta\in L^{2}_{\calF_{\t}}(\Omega,H)$ we define
\begin{align*}
g:=\zeta-\int_{0}^{\t}T(\t-s)\mathscr{M}(\zeta)dW(s)\in L^{2}_{\calF_{\t}}(\Omega,H).
\end{align*}  Then there exists
$u\in L^2_{\F}(0,\t;U)$ such that
 $\tilde{\Phi}_\t u=g$. Thus
\begin{align*}
\Phi^{W}_{\tau}u-\zeta=\int_{0}^{\t}T(\t-s)\mathscr{M}(\Phi^{W}_{s}u-\zeta)dW(s).
\end{align*}
Using It\^{o}'s isometry and Gronwall's inequality , we deduce $\E\|\Phi^{W}_{\tau}u-\zeta\|^{2}_{H}=0$.
This implies that $\Phi^{W}_{\tau}u=\zeta$. By remark \ref{exact-controlla-stoch}, the system \eqref{ABC-uy} is $\t$-exactly controllable.\\ (iii) The exact controllability of the parabolic and hyperbolic stochastic equations has recently been studied in \cite{Lu-Zhang-21}. The authors offered examples of the lack of exact controllability that can arise. On the other hand, the authors proved the exact controllability of stochastic transport equations and some typical parabolic equations.
\end{remark}

In the rest of this section we focus on the properties of the output process $Y(\cdot)$ of the system \eqref{ABC-uy}. By analogy to the deterministic well-posed systems, we adopt the following definition.
\begin{definition}\label{Def-Well-posed-ABCW}
 We say that the stochastic system \eqref{ABC-uy} is well-posed if there exists
 an extension $\tilde{C}:D(\tilde{C})\subset H\to \mathscr{Y}$ of the operator $C$ such that for any $\xi\in L^2_{\calF_0}(\Om,H)$ and $u\in L^2_{\F}(0,\al;U)$
 ($\al>0$), the mild solution  satisfies
 $X(t;\xi,u)\in D(\tilde{C})$ for almost every $t> 0$ and $\P$-a.s. and
\begin{align*}
\|\tilde{C}X(\cdot)\|_{L^2_{\F}(0,\al;\mathscr{Y})}\le c \left(\|\xi\|_{L^2_{\calF_{0}}(\Om,H)}+\|u\|_{L^2_{\F}(0,\al;U)}\right)
\end{align*} where $c:=c_\alpha>0$ is a constant.
\end{definition}
To investigate the well-posedness of the system \eqref{ABC-uy}, we first prove the following results on stochastic convolution.
\begin{proposition}\label{fondamental-lemma}
Assume that $C$ is an admissible observation operator for $A$. For  any $\al>0$ and $\zeta\in L^2_{\F}(0,\al;H)$, we have
\begin{align*}
& (T\diamond \zeta)(t):=\int_{0}^{t}T(t-s)\zeta(s)dW(s)\in D(C_{\Lambda}),\quad a.e.\, t\ge 0,\quad \mathbb{P}-a.s, \quad\text{and,}\cr & \mathbb{E}\int_{0}^{\al}\|C_{\Lambda}(T\diamond \zeta)(t)\|^{2}_{\mathscr{Y}} dt\leq \ga^2\mathbb{E}\int_{0}^{\al}\|\zeta(s)\|^{2}_{H}ds\end{align*}
for some constant $\ga:=\ga(\al)>0$ independent of $\zeta$.
\end{proposition}
\begin{proof}
Assume that $C$ is admissible for $A$ and denote by $\ga>0$ the associated admissibility constant. Let $\la\in\rho(A)$ and define the bounded operator $\mathscr{C}_\la:=C\la R(\la,A)$. Let $\al>0$ and $\zeta\in L^2_{\F}(0,\al;H)$. Using It\^{o}'s isometry, Fubini's Theorem, a change of variables and the admissibility of $C$ for $A,$ we obtain
\begin{align*}
\mathbb{E}\int_{0}^{\al}\left\|\mathscr{C}_\la \int_{0}^{t}T(t-s)\zeta(s)dW(s)\right\|^{2}_{\mathscr{Y}}dt&=\int_{0}^{\al}\mathbb{E}\|\int_{0}^{t}\mathscr{C}_\la T(t-s)\zeta(s)dW(s)\|^{2}_{\mathscr{Y}}dt\\ &=\int_{0}^{\al}\mathbb{E}\int_{0}^{t}\|CT(t-s)\lambda R(\lambda,A)\zeta(s)\|^{2}_{\mathscr{Y}}dsdt\\&\leq \mathbb{E}\int_{0}^{\al}\int_{s}^{\al}\|CT(t-s)\lambda R(\lambda,A)\zeta(s)\|^{2}_{\mathscr{Y}}dtds\\&\leq \mathbb{E}\int_{0}^{\al}\int_{0}^{\al-s}\|CT(\tau)\lambda R(\lambda,A)\zeta(s)\|^{2}_{\mathscr{Y}}d\tau ds\\&\leq \mathbb{E}\int_{0}^{\al}\int_{0}^{\al}\|CT(\tau)\lambda R(\lambda,A)\zeta(s)\|^{2}_{\mathscr{Y}}d\tau ds\\
&\leq\gamma^{2}\mathbb{E}\int_{0}^{\al}\|\lambda R(\lambda,A)\zeta(s)\|^{2}_{H} ds.
\end{align*}
Then for $n,m\in\mathbb{N}$ such that $n,m>\om_0(A),$ we have
$$\mathbb{E}\int_{0}^{\al}\|(\mathscr{C}_n -\mathscr{C}_m)(T\diamond f)(t)\|^{2}_{\mathscr{Y}}dt\leq \gamma^{2}\mathbb{E}\int_{0}^{\al}\|(n R(n,A)-m R(m,A))\zeta(s)\|^{2}_{H} ds.$$
This shows that $(\mathscr{C}_n(T\diamond \zeta))_{n\in\mathbb{N}^{*}}$ is a Cauchy sequence on $L^{2}_{\F}(0,\al;\mathscr{Y})$. So we can extract a subsequence $(\mathscr{C}_{n_{k}}(T\diamond \zeta)(t))_{k\in\N}$ which converges for a.e. $t\in [0,\al]$ and $\mathbb{P}$-a.s. Thus $(T\diamond \zeta)(t)\in D(C_{\Lambda})$ for a.e. $t\in [0,\al]$ and $\mathbb{P}$-a.s. The rest of the proof follows immediately from the first inequality above.
\end{proof}
\begin{remark}
We assume that the semigroup $\T$ is analytic on $H$ and define the operator $M=(-A)^{\theta}$ for some $\theta\in (0,\frac{1}{2})$. We denote by $C$ the restriction of $M$ to $D(A)$. Using the analyticity of the semigroup we have $\|CT(t)\|\le \frac{M}{t^{\theta}}$ for $t>0$. This implies that $C:D(A)\to X$ is an admissible observation operator for $A$. On the other hand, by using the same arguments as in \cite[Example 4]{ABDH-20}, one  can see that the Yosida extension of $C$ for $A$ coincides with the operator $(-A)^\theta$. Thus for analytic semigroup Proposition \ref{fondamental-lemma} holds also if we replace $C_\Lambda$ by $(-A)^\theta$ for any $\theta\in (0,\frac{1}{2})$. Moreover, if $\T$ is a contractive analytic semigroup of type less than $\frac{\pi}{2},$ one can use Le Merdy \cite{LeMery} to include the critical case $\theta=\frac{1}{2}$ in the discussion above. These kind of results are already available in stochastic analysis for parabolic evolution equations , see e.g. \cite[Chap.4, Section 14.2]{Da-Za-Book-96}, \cite{flandoli1}. Here we have only offered a different proof for this particular case. However, in Proposition \ref{fondamental-lemma}, the semigroup $\T$ is general and the operator $C$ is neither closed nor closeable (same things for the Yosida extension $C_\Lambda$).
\end{remark}

The second main result of this section is the following:
\begin{theorem}\label{ABCW-wellposed}
Assume that the system $(A,B,C)$ is regular on $X,U,\mathscr{Y}$. Let $C_\Lambda$ be the Yosida extension of $C$ for $A$ and $\al>0$ be arbitrary. Then the mild solution of the system \eqref{ABC-uy} satisfies
$X(t;\xi,u)\in D(C_\Lambda)$ for almost every $t> 0$ and $\P$-a.s. and
\begin{align*}
\|C_\Lambda X(\cdot)\|_{L^2_{\F}(0,\al;\mathscr{Y})}\le c \left(\|\xi\|_{L^2_{\calF_{0}}(\Om,H)}+\|u\|_{L^2_{\F}(0,\al;U)}\right)
\end{align*} where $c:=c_\al>0$ is a constant. In particular the system \eqref{ABC-uy} is well-posed.
\end{theorem}
\begin{proof}
Let $\al>0$ be arbitrary, $u\in L^2_{\F}(0,\al;U)$ and $\xi\in L^2_{\calF_0}(\Om,H)$. According to Theorem \ref{thm-existence}, the mild solution of the stochastic system \eqref{ABC-uy} exists, unique and it satisfies
\begin{align*}
X(t;\xi,u)=X(t,\xi,0)+\Phi^W_t u,\qquad t\ge 0,
\end{align*}
where $\Phi^{W}_t u$ and $X(t,\xi,0)$ are given by \eqref{stoch-maps} and \eqref{walid-hadd}, respectively. By using It\^{o}'s Isometry and Gronwall's inequality, one can see that for $t\in [0,\al],$
 \begin{align}\label{just-here}
 \E\|X(t,\xi,0)\|^2\le c_1 \|\xi\|^2_{L^2_{\calF_{0}}(\Om,H)}
 \end{align}
 for a constant $c_1:=c_1(\al)>0$. As $C$ is admissible for $A,$ by Subsection \ref{sec:RLS} and Proposition \ref{fondamental-lemma} we have
 $X(t,\xi,0)\in D(C_\Lambda)$ for a.e. $t\ge 0$ and $\P$-a.s. Moreover, by using Proposition \ref{fondamental-lemma} and the inequality \eqref{just-here}, we obtain
 \begin{align}\label{one}
\int^\al_0 \E \left\|C_\Lambda X(t;\xi,0)\right\|^2_{\mathscr{Y}}dt\le c_2^2 \|\xi\|^2_{L^2_{\calF_{0}}(\Om,H)}
\end{align}
for a constant $c_2:=c_2(\al)>0$. In addition, as the system $(A,B,C)$ is regular, we have
\begin{align*}
\int^\al_0 \|C_\Lambda \int^t_0 T_{-1}(t-s)Bu(s,\om)ds\|^2_{\mathscr{Y}}dt\le \kappa_1\int^\al_0 \|u(s,\om)\|^2ds,\qquad \P-a.s.,
\end{align*}
for a constant $\kappa_1:=\kappa_1(\al)>0$ independent of $\om$. Thus
\begin{align}\label{s}
\E\int^\al_0 \left\|C_\Lambda \int^t_0 T_{-1}(t-s)Bu(s,\om)ds\right\|^2_{\mathscr{Y}}dt\le \kappa_1\|u\|^2_{L^2_{\F}(0,\al;U)}.
\end{align}
 Moreover, Proposition  \ref{fondamental-lemma} and the inequality \eqref{phi1}, we have
\begin{align}\label{g1}
\begin{split}
& \int^t_0 T(t-s)\mathscr{M}\left(\Phi^{W}_s u\right)dW(s)\in D(C_\Lambda),\; a.e.\, t\ge 0,\;\P-a.s.,\quad\text{and,}\cr
&\E\int^\t_0\left\| C_\Lambda\int^t_0 T(t-s)\mathscr{M}\left(\Phi^{W}_s u\right)dW(s)\right\|^2_{\mathscr{Y}}dt
 \le \kappa_2^{2} \|u\|^{2}_{L^2_{\F}(0,\al;U)}
 \end{split}
\end{align}
for a constant $\kappa_2:=\ga(\al) \tilde{\ga}(\al) \|\mathscr{M}\|>0$. Hence by combining \eqref{s} and \eqref{g1}, we obtain
 $\Phi^{W}_t u\in D(C_\Lambda)$ for almost every $t\ge 0$ and
\begin{align}\label{ggg}
\int^\al_0 \E\|C_\Lambda\Phi^{W}_t u\|^2_{\mathscr{Y}}dt \le c_2^2 \,\|u\|_{L^2_{\F}(0,\al;U)}^2
\end{align}
for a constant $c_2:=c_2(\al)>0$. Finally, the result immediately follows from \eqref{one} and \eqref{ggg}.
\end{proof}
\begin{example}\label{Schro-regular}
Let $\mathscr{O}\subset \R^n$ ($n\ge 2$) be an open bounded region with $C^3$-boundary
$\partial\mathscr{O}=\overline{\Gamma}_0\cup \overline{\Gamma}_1$,
 where $\Gamma_0$ and $\Gamma_1$ are disjoint
parts of the boundary relatively open in $\partial\mathscr{O}$ and ${\rm int}(\Gamma_0)\neq \mathscr{O}$. Consider the following stochastic Schr\"odinger
equation with partial Dirichlet control and collocated observation
\begin{align}\label{chro}
\begin{cases}
dX(t,x)=i\Delta X(t,x)dt+q(x)X(t,x)dW(t),
& t> 0,\;x\in\mathscr{O},\cr X(0,x)=\xi(x),& x\in \mathscr{O},\cr X(t,x)=0,
& t\ge 0,\; x\in \Gamma_1,\cr X(t,x)=u(t,x), & t\ge 0,\; x\in\Gamma_0,\cr
Y(t,x)=i\displaystyle\frac{\partial (\Delta^{-1}X)}{\partial \nu},& t\ge 0,\; x\in \Gamma_0,
\end{cases}
\end{align}
where
$\nu$ is the unit normal of $\partial\mathscr{O}$ pointing towards the exterior of $\mathscr{O}$, and $q\in L^\infty(\mathscr{O})$.  Let $H=H^{-1}(\mathscr{O})$ (the dual space
of the Sobolev space $H^1_0(\mathscr{O})$ with respect to the pivot
space $L^2(\mathscr{O})$) and $U=L^2(\Gamma_0)$.
According to \cite[Theorem 1.2]{Guo-sha}, the deterministic system assocaited with \eqref{chro} is regular. Thus, by Theorem \ref{ABCW-wellposed}, the stochastic
system \eqref{chro} is well-posed.
\end{example}

\section{Admissibility to perturbed abstract stochastic linear systems}\label{sec:new-VCF}
The main purpose of this section is to introduce a new variation of constants formula to the solutions of the perturbed abstract stochastic linear systems \eqref{perturbed-S-Cauchy}. Moreover, we prove a result on the well-posedness of perturbed boundary stochastic linear systems
\subsection{A new variation of constants formula for perturbed abstract stochastic linear systems}
Let consider  the operator
\begin{align}\label{perturbed-generator}
\mathscr{A}=A+\mathscr{P},\quad D(\mathscr{A})=D(A).
\end{align}
In this section, we denote by $\mathscr{P}_{\Lambda}$ the Yosida extension of $\mathscr{P}$ for $A$. The following result is summarized from \cite[Theorem 2.1]{Hadd-SF} and \cite{Ha-Id-SCL}.
\begin{theorem}\label{M-V-thm}
Assume that $\mathscr{P}\in\calL(D(A),H)$ is an admissible observation operator for $A$. then the operator $(\mathscr{A},D(\mathscr{A}))$ generates a $C_0$-semigroup $\mathscr{T}:=(\mathscr{T}(t))_{t\ge 0}$ on $H$ satisfying $\mathscr{T}(t)x\in D(\mathscr{P}_\Lambda)$ for a.e. $t>0$ and all $x\in H$. Moreover,
\begin{align}\label{semigroup-S}
\begin{split}
&\int^\al_0 \|\mathscr{P}_\Lambda \mathscr{T}(\si)x\|^2d\si\le \tilde{\ga}^2 \|x\|^2,\cr
&\mathscr{T}(t)x=T(t)x+\int_{0}^{t}T(t-s)\mathscr{P}_\Lambda \mathscr{T}(s)xds
\end{split}
\end{align}
for any $t\ge 0$ and $x\in H$, where $\al>0$ and $\tilde{\ga}:=\tilde{\ga}(\al)>0$ are constants.
\end{theorem}
\begin{remark}\label{remark-section3}
The estimation in  \eqref{semigroup-S} says that $\mathscr{P}$ is an admissible observation operator for $\mathscr{A}$. More generally, it is shown in \cite{Ha-Id-SCL} that the conditions if in addition $C\in\calL(D(A),Y)$  is an admissible observation operator for $A$, then $C$ is also admissible for $\mathscr{A}$. Moreover, if $C_{\Lambda}$ and $\tilde{C}_{\Lambda}$ denote the Yosida extensions of $C$ for $A$ and $\mathscr{A},$ respectively, then $C_\Lambda\equiv \tilde{C}_{\Lambda}$ on $D(C_{\Lambda})\cap D(\mathscr{P}_{\Lambda})$. In particular, $\mathscr{P}_{\Lambda}\equiv \tilde{\mathscr{P}}_{\Lambda},$ where $\tilde{\mathscr{P}}_{\Lambda}$ denotes the Yosida extension of $\mathscr{P}$ for $\mathscr{A}$.
\end{remark}
Now we have the following result.
\begin{proposition}\label{prop-T-L-W}
Assume that $\mathscr{P}\in\calL(D(A),H)$ is an admissible observation operator for $A$, and let $\mathscr{T}$ the $C_0$-semigroup on $H$ generated by $\mathscr{A}$. Then the stochastic abstract Cauchy problem \eqref{perturbed-S-Cauchy} has a unique mild solution $X(\cdot)\in \calC_\F(0,+\infty;L^2(\Om,H))$ satisfying
\begin{align}\label{sol-mild}
\begin{split}
& X(t)=\mathscr{T}(t)\xi+\int^t_0 \mathscr{T}(t-s)\mathscr{M}(X(s))dW(s),\qquad t\ge 0, \cr & X(t)\in D(\mathscr{P}_{\Lambda}),\; a.e.\, t\ge 0,\; \mathbb{P}-a.s, \quad \E \int^\al_0 \|\mathscr{P}_{\Lambda} X(t)\|^2dt\le c^2 \E\|\xi\|^2
 \end{split}
\end{align}
for constants $\al>0$ and $c=c(\al)>0$.
\end{proposition}
\begin{proof}
The existence of the mild solution that satisfies the variation of constants formula in \eqref{sol-mild} follows immediately from Theorem \ref{M-V-thm}, while the estimation in \eqref{sol-mild}  is obtained by combining Proposition \ref{fondamental-lemma} with Remark \ref{remark-section3}.
\end{proof}
In the sequel we are interested in rewriting  the mild solution obtained in Proposition \ref{prop-T-L-W} using only the initial semigroup $\T$ and
the Yosida extension of $ \mathscr{P}$ for $A$. To this end, we need the following technical results.
\begin{lemma}\label{technical-lemma}
Assume that $\mathscr{P}$ is an admissible observation operator for $A$ and let $X(\cdot)$ be the process given in \eqref{sol-mild}.
For any $n>\om_0(\mathscr{A})$, we  set $\mathscr{M}_n=n R(n,\mathscr{A})\mathscr{M}$ and define
\begin{align}\label{suite}
X^n(t)=\mathscr{T}(t)\xi+\int^t_0 \mathscr{T}(t-s)\mathscr{M}_n(X^n(s))dW(s),\quad t\ge 0.
\end{align}
For any $\al>0$, $X^n(\cdot)$ converges to $X(\cdot)$ in $\calC_{\F}(0,\al;L^2(\Om,H))$ as $n\to+\infty$. Moreover, for any $n>\om_0(\mathscr{A}),$ we have $X^n(t)\in D(\mathscr{P}_\Lambda)$ for a.e. $t\ge 0,$ $\P$-a.s., and
\begin{align}\label{PPP}
\lim_{n\to +\infty}\|\mathscr{P}_\Lambda X^n(\cdot)-\mathscr{P}_\Lambda X(\cdot)\|_{L^2_{\F}(0,\al;H)}=0.
\end{align}
\end{lemma}
\begin{proof}
The fact that $X^n(\cdot)$ converges to $X(\cdot)$ in $\calC_{\F}(0,\al;L^2(\Om,H))$ as $n\to+\infty$ follows in the same way as in \cite[Theorem 3.4]{govi-appr-book}. Given $\al>0,$ let us first prove that
\begin{align}\label{AAAAA}
\lim_{n\to \infty}\E\int^\al_0\|\mathscr{M}_n(X^n(s))-\mathscr{M}(X(s))\|^2ds=0.
\end{align}
To this end, let $n>\om_0(\mathscr{A})$ arbitrary sufficiently large and $M'\in \R^+$ such that $\|nR(n,A)\|\le M'$. We have
\begin{align*}
&\E\int^\al_0\|\mathscr{M}_n(X^n(s))-\mathscr{M}(X(s))\|^2ds\cr & \qquad \le
\int^\al_0 \E \|\mathscr{M}_n(X^n(s)-X(s))\|^2ds+\int^t_0 \E \|\mathscr{M}_n(X(s))-\mathscr{M}(X(s))\|^2ds
\cr & \qquad \le
M' \|\mathscr{M}\|\int^\al_0 \E \|X^n(s)-X(s)\|^2ds+\int^\al_0 \E \|\mathscr{M}_n(X(s))-\mathscr{M}(X(s))\|^2ds.
\end{align*}
On the other hand, we have  $\mathscr{M}_n(X(s))\to \mathscr{M}(X(s))$ as $n\to+\infty$  and
$\|\mathscr{M}_n(X(s))\|\le M' \|\mathscr{M}\| \|X(s)\|$ for any $s\in [0,\al]$ and $\P$-a.s.. Now  \eqref{AAAAA} follows from the first assertion in the proof and the Dominated convergence theorem in $L^2$-spaces. Now let $\tilde{\mathscr{P}}_\Lambda$ be the Yosida extension
of $\mathscr{P}$ for $\mathscr{A}$. By Remark \ref{remark-section3}, we have $\tilde{\mathscr{P}}_\Lambda=\mathscr{P}_\Lambda$.
Now by using Proposition \ref{fondamental-lemma}, $X^n(t)\in D(\mathscr{P}_\Lambda)$ for a.e. $t\ge 0,\;\P$-a.s. and
\begin{align*}
&\E\int^\al_0 \|\mathscr{P}_\Lambda X^n(t)-\mathscr{P}_\Lambda X(t)\|^2dt\cr &\qquad\qquad=
\E\int^\al_0\left\|\mathscr{P}_{\Lambda}\int^t_0 \mathscr{T}(t-s)(\mathscr{M}_n(X^n(s))-\mathscr{M}(X(s)))dW(s)\right\|^2 dt
\cr &\qquad\qquad\leq \tilde{\ga}^2
\E\int^\al_0\|\mathscr{M}_n(X^n(s))-\mathscr{M}(X(s))\|^2ds.
\end{align*}
for a constant $\tilde{\ga}:=\tilde{\ga}(\al)>0$ independent of $n$. Thus \eqref{PPP} follows by \eqref{AAAAA}.
\end{proof}
The following gives another expression to the process $X^n(\cdot)$.
\begin{lemma}\label{new variation0}
Assume that $\mathscr{P}$ is an admissible observation operator for $A$  and let $(X^n(\cdot))_{n>\om_0(A)}$ the process given by \eqref{suite}. For any $ \xi \in L^{2}_{\mathcal{F}_{0}}(\Omega;H)$, we have
\begin{align}\label{VCF-MV0}
X^n(t)=T(t)\xi+\int_{0}^{t}T(t-s)\mathscr{P}_{\Lambda}X^n(s)ds+\int_{0}^{t}T(t-s)\mathscr{M}_n(X^n(s))dW(s)\end{align}
for any $t\ge 0$ and $\P$-a.s.
\end{lemma}
\begin{proof}
By using \eqref{semigroup-S}, we have
\begin{align}\label{GGGG}
 X^n(t)&=T(t)\xi+ \int_{0}^{t}T(t-s)\mathscr{M}_n(X^n(s))dW(s)+\int_{0}^{t}T(t-s)\mathscr{P}_{\Lambda}\mathscr{T}(s)\xi ds
 \\&\hspace{3cm}+\int_{0}^{t}\int_{0}^{t-s}T(t-s-\tau)\mathscr{P}_{\Lambda}\mathscr{T}(\tau)\mathscr{M}_n(X^n(s)) d\tau dW(s).
\end{align}
By a change of variables and Fubini's stochastic theorem,
\begin{align*}
\int_{0}^{t}\int_{0}^{t-s}T(t-s-\tau)&\mathscr{P}_{\Lambda}\mathscr{T}(\tau)\mathscr{M}_n(X^n(s)) d\tau dW(s)\cr
& =\int_{0}^{t}\int_{0}^{t-s}T(t-s-\tau)\mathscr{P}nR(n,\mathscr{A})\mathscr{T}(\tau)\mathscr{M}(X^n(s)) d\tau dW(s)\cr
&=\int_{0}^{t}\int_{s}^{t}T(t-\si)\mathscr{P}nR(n,\mathscr{A})\mathscr{T}(\si-s)\mathscr{M}(X^n(s)) d\si dW(s)
\\&=\int_{0}^{t}T(t-\si)\mathscr{P}nR(n,\mathscr{A})\int^{\si}_{0}\mathscr{T}(\si-s)\mathscr{M}(X^n(s)) dW(s) d\si
\\&=\int_{0}^{t}T(t-\si)\mathscr{P}_\Lambda nR(n,\mathscr{A})\int^{\si}_{0}\mathscr{T}(\si-s)\mathscr{M}(X^n(s)) dW(s) d\si
\\&=\int_{0}^{t}T(t-\si)\mathscr{P}_\Lambda \int^{\si}_{0}\mathscr{T}(\si-s)nR(n,\mathscr{A})\mathscr{M}(X^n(s)) dW(s) d\si
\\&=\int_{0}^{t}T(t-\si)\mathscr{P}_\Lambda \int^{\si}_{0}\mathscr{T}(\si-s)\mathscr{M}_n(X^n(s)) dW(s) d\si.
\end{align*}
Now replacing this expression in \eqref{GGGG} and using \eqref{semigroup-S}, the formula \eqref{VCF-MV0} follows.
\end{proof}
\begin{theorem}\label{new variation}
Assume that $\mathscr{P}$ is an admissible observation operator for $A$  and let $X:[0,+\infty)\times \Om\to H$ be the process given in \eqref{sol-mild}. For any
 $ \xi \in L^{2}_{\mathcal{F}_{0}}(\Omega,H)$,
\begin{align}\label{VCF-MV} X(t)=T(t)\xi+\int_{0}^{t}T(t-s)\mathscr{P}_{\Lambda}X(s)ds+\int_{0}^{t}T(t-s)\mathscr{M}(X(s))dW(s)\end{align}
for any $t\ge 0$ and $\P$-a.s.
\end{theorem}
\begin{proof}
Let $X(\cdot)$ and $X^n(\cdot),\;n>\om_0(A),$  be the processes  given by \eqref{suite} and \eqref{sol-mild}, respectively.
Let $\al>0$ be arbitrary and $ \xi \in L^{2}_{\mathcal{F}_{0}}(\Omega;H)$.
According to Proposition
\ref{prop-T-L-W}, the following process
\begin{align*}
Z(t):=T(t)\xi+\int_{0}^{t}T(t-s)\mathscr{P}_{\Lambda}X(s)ds+\int_{0}^{t}T(t-s)\mathscr{M}(X(s))dW(s),
\end{align*}
for $t\ge 0,\; \P-a.s.,$ is well defined. Let $\delta>\om_0(A)$ and $M\ge 1$ such that $\|T(t)\|\le M e^{\delta t}$
for any $t\ge 0$. By using the same arguments as in the proof of Lemma \ref{new variation0}, H\"older's inequality and It\^{o}'s isometry, we obtain
\begin{align*}
&\E \|X^n(t)-Z(t)\|^2 \cr &\le c(\al)\left( \E \int^\al_0 \|\mathscr{P}_\Lambda X^n(s)-\mathscr{P}_\Lambda X(s)\|^2ds+
\E \int^\al_0\|\mathscr{M}_n(X^n(s))-\mathscr{M}(X(s))\|^2ds\right)
\end{align*}
for any $t\in [0,\al]$, where $c(\al)=2 M e^{2|\delta|\al}$ is a constant. According to \eqref{PPP} and \eqref{AAAAA},
\begin{align*}
\lim_{n\to +\infty}\E \|X^n(t)-Z(t)\|^2=0
\end{align*}
Thus, by Lemma \ref{technical-lemma}, we obtain $X(t)=Z(t)$ for any $t\in [0,\al]$ and $\P$-a.s. As $\al$ is arbitrary we have $X(\cdot)=Z(\cdot)$ on $[0,+\infty)$, $\P$-a.s.
This ends the proof.
\end{proof}
\subsection{Perturbed stochastic system with boundary control and point observation}
Assume that there exists a Hilbert space $\mathscr{Z}$ such that $\mathscr{Z}\subset H$ with continuous and dense embedding. Let $\sigma: \mathscr{Z}\to U$ and $M:\mathscr{Z}\to \mathscr{Y}$ be linear operators such that $\si$ is surjective and $M$ not necessary closed. In addition consider a closed "maximal" operator $A_m:\mathscr{Z}\to H$ and $K:\mathscr{Z}\to H$ be a linear operator (not necessary closed). We consider the perturbed stochastic system with boundary control and point observation defined by
\begin{align}\label{P-ABC}
\begin{cases}
dX(t)=((A_m+K)X(t))dt+\mathscr{M}(X(t))dW(t),& t>0,\quad X(0)=\xi,\cr \si X(t)=u(t),& t\ge 0,\cr Y(t)=M X(t),& t\ge 0.
\end{cases}
\end{align}
We assume that $A=(A_m)_{|D(A)}$, where $D(A)=\ker\si$ is the generator of a $C_0$-semigroup $\T=(T(t))_{t\ge 0}$ on $H$. Without loss of generality, we assume that $0\in\rho(A)$. Let $D\in\calL(U,H)$ be the Dirichlet operator associated with $\si$ and $A_m$, which is the inverse of the restection of $\si$ to $\ker(A_m)$ (see e.g. \cite{Grei}). We define the following operators  $B=(-A_{-1})D\in\calL(U,H_{-1})$, $\mathscr{P}=K_{|D(A)}\in\calL(D(A),H),$ and $C:=M_{|D(A)}\in\calL(D(A),\mathscr{Y})$. Moreover, we denote by $\mathscr{P}_\Lambda$ and $C_\Lambda$ the Yosida extensions of $\mathscr{P}$ and $C$ for $A,$ respectively.

The following result shows the well-posedness of the system \eqref{P-ABC}.
\begin{theorem}\label{bounadry-theorem}
Assume that the systems $(A,B,\mathscr{P})$ and $(A,B,C)$ are regular. Then the system \eqref{P-ABC} has a unique mild solution $X(\cdot)\in\calC_{\F}(0,+\infty;L^2(\Om,H))$ such that $X(t)\in D(\mathscr{P}_\Lambda)\cap D(C_\Lambda)$ for a.e. $t>0$ and
\begin{align}\label{C-estimate}
\|C_\Lambda X(\cdot)\|_{L^2_{\F}(0,\al;\mathscr{Y})}\le c \left(\|\xi\|_{L^2_{\calF_0}(\Om,H)}+\|u\|_{L^2_{\F}(0,\al;U)}\right)
\end{align}
for any $\xi\in L^2_{\calF_0}(\Om,H)$ and $u\in L^2_{\F}(0,\al;U)$, where $c:=c(\al)>0$ is a constant.
\end{theorem}
\begin{proof}
According to \cite{Grei}, we have $\mathscr{Z}=D(A)\oplus\ker(A_m)$. This implies that $A_m$ coincides with $A_{-1}+B\si$ on $\mathscr{Z}$. Thus the stochastic equation in \eqref{P-ABC} takes the following form
\begin{align}\label{Distr}
dX(t)=\left[(A+K)X(t)+Bu(t)\right]dt+\mathscr{M}(X(t))dW(t),\quad t>0,\quad X(0)=\xi.
\end{align}
Let $\mathscr{A}$ be the generator of the semigroup $(\mathscr{T}(t))_{t\ge 0}$ as in Theorem \ref{M-V-thm}. According to \cite{MP-10}, the fact that system $(A,B,\mathscr{P})$ is regular implies that $B$ is an admissible control operator for $\mathscr{A}$. Thus the equation \eqref{Distr} has a unique mild solution $X(\cdot)\in \calC_{\F}(0,+\infty,L^2(\Om,H))$ satisfying
\begin{align*}
X(t)=\mathscr{T}(t)\xi+\int^t_0\mathscr{T}_{-1}(t-s)Bu(s)ds+\int^t_0 \mathscr{T}(t-s)\mathscr{M}(X(s))dW(s)
\end{align*}
for any $t\ge 0$ and $u\in L^2_{\F}(0,+\infty;U)$. According to \cite{MP-10}, the systems $(\mathscr{A},B,\mathscr{P})$ and $(\mathscr{A},B,C)$ are regular. Thus by combining  Remark \ref{remark-section3} and Theorem \ref{ABCW-wellposed}, we have $X(t)\in D(\mathscr{P}_\Lambda)\cap D(C_\Lambda)$ for a.e. $t>0$ and the estimation \eqref{C-estimate} holds. Hence the system \eqref{P-ABC} is well-posed.
\end{proof}
\begin{example}
Let $\mathscr{O}$, $\Gamma_0$ and $\Gamma_1$ as in Example \ref{Schro-regular} and let $k:\overline{\mathscr{O}}\times \partial\mathscr{O}\to \R$ be a continuous function. Consider the following Schr\"odinger stochastic input-output system with non-local integral term:
\begin{align}\label{chro1}
\begin{cases}
dX(t,x)=\left(i\Delta X(t,x)+i\displaystyle\int_{\Gamma_1}k(x,y)X(t,y)dy\right)dt\cr \hspace{5cm}+q(x)X(t,x)dW(t),
& t> 0,\;x\in\mathscr{O},\cr X(0,x)=\xi(x),& x\in \mathscr{O},\cr X(t,x)=0,
& t\ge 0,\; x\in \Gamma_0,\cr \frac{\partial}{\partial\nu}X(t,x)=u(t,x), & t\ge 0,\; x\in\Gamma_1,\cr
Y(t,x)=iX(t,x),& t\ge 0,\; x\in \Gamma_1,
\end{cases}
\end{align}
where
$\nu$ is the unit normal of $\partial\mathscr{O}$ pointing towards the exterior of $\mathscr{O}$, and $q\in L^\infty(\mathscr{O})$. In order to reformulate the system \eqref{chro1} as the abstract system \eqref{P-ABC}, we use the following notation: Let the state space $H=L^2(\mathscr{O})$, the intermediate space $\mathscr{Z}=H^2(\mathscr{O}),$ and the control space $U=L^2(\partial\mathscr{O})$. In addition, we consider
\begin{align*}
A=i\Delta,\quad D(A)=\left\{g\in \mathscr{Z}: g_{|\Gamma_0}=\frac{\partial g}{\partial\nu}|_{\Gamma_1}=0\right\}.
\end{align*}
This operator is skew-adjoint and generates an unitary group $\T$ on $H$. We consider the operators $\mathscr{M}:H\to H,$ $G:\mathscr{Z} \to U$ and $M:\mathscr{Z} \to U$ defined by
\begin{align*}
\mathscr{M}h=q(\cdot)h,\quad G g=\begin{cases} g,& \text{on}\;\Gamma_0,\cr
\frac{\partial g}{\partial\nu},&\text{on}\;\Gamma_1,
\end{cases}\qquad M g=\begin{cases} 0,& \text{on}\;\Gamma_0,\cr
ig,&\text{on}\;\Gamma_1.
\end{cases}
\end{align*}
On the other hand, we define an operator $K:\mathscr{Z}\to H$ by
\begin{align*}
K:=\mathbb{L} M\quad\text{with}\quad (\mathbb{L} g)(x)=\int_{\partial\mathscr{O}}k(x,y)g(y)dy.
\end{align*}
Let $D\in\calL(U,H)$ be the Dirichlet operator associated with the system \eqref{chro1}. This is
\begin{align*}
h=Dv \Longleftrightarrow \Delta h=0,\quad h_{|\Gamma_0}=0,\quad \frac{\partial h}{\partial\nu}{|\Gamma_1}=v.
\end{align*}
We set $\A=\Delta$ with domain $D(\A)=D(A)$. Then $\A$ is a generator of a $C_0$-semigroup on $H$. We denote by $\A_{-1}$ its extension to $H_{-1},$ which is a generator on $H_{-1}$. We denote by $B=-i\A_{-1}D$, $C=M_{|D(A)}$ and $\mathscr{P}:=K_{|D(A)}=\L C$. A routine calculus as in \cite{LT} and \cite{Guo-sha}, shows that the system $(A,B,C)$ is regular, and hence the system $(A,B,\mathscr{P})$ is also regular. Now Theorem \ref{bounadry-theorem} implies that the stochastic system \eqref{chro1} is well-posed.
\end{example}
The following result gives an expression of the mild solution of the system \eqref{P-ABC} using the semigroup $\T$.
\begin{proposition}\label{BP}
Assume that the systems $(A,B,\mathscr{P})$ and $(A,B,C)$ are regular and let $X(\cdot)$ be the mild solution of the boundary stochastic system \eqref{P-ABC}. Then
\begin{align}\label{New-formula-B}
\begin{split}
X(t)=T(t)\xi+\int^t_0 T(t-s)\mathscr{P}_{\Lambda}X(s)ds&+\int^t_0 T_{-1}(t-s)Bu(s)ds\cr &+\int^t_0 T(t-s)\mathscr{M}(X(s))dW(s)
\end{split}
\end{align}
for any $t\ge 0,$ $\xi\in L^2_{\calF_0}(\Om,H)$ and $u\in L^2_{\F}(0,+\infty;U)$.
\end{proposition}
\begin{proof}
We proceed by approximation. Let $(\mathfrak{A},D(\mathfrak{A}))$ be the operator defined in \eqref{lax-gen}. Remak that if $\left(\begin{smallmatrix}x\\g\end{smallmatrix}\right)\in D(\mathfrak{A})$, in particular we have $A_{-1}+Bg(0)=A_{-1}(x-Dg(0))\in H$. Thus $x-Dg(0)\in D(A)\subset \mathscr{Z},$ and then $x\in \mathscr{Z}$. Define
\begin{align*}
\calK:=\begin{pmatrix}K& 0\\ 0&0\end{pmatrix}:D(\mathfrak{A})\to \mathfrak{X},\qquad \varrho(t)=\left(\begin{smallmatrix}X(t)\\S^U(t)u\end{smallmatrix}\right).
\end{align*}
As in the proof of Theorem \ref{thm-existence}, the equation \eqref{Distr} is reformulated as the following perturbed stochastic evolution equation on $\mathfrak{X,}$
\begin{align}\label{Distr-Big}
d\varrho(t)=\left[(\mathfrak{A}+\calK)\varrho(t)\right]dt+\mathscr{B}(\varrho(t))dW(t),\quad t>0,\quad \varrho(0)=\left(\begin{smallmatrix}\xi\\u\end{smallmatrix}\right).
\end{align}
Next we prove that $\calK$ is an admissible observation operator for $\mathfrak{A}$. In fact, by taking $\left(\begin{smallmatrix}x\\g\end{smallmatrix}\right)\in D(\mathfrak{A})$ and using \eqref{lax-sem}, we obtain
\begin{align}\label{nnn}
\int^{t_0}_0\left\|\calK\mathfrak{T}(t)\left(\begin{smallmatrix}x\\g\end{smallmatrix}\right)\right\|^2_{\mathfrak{X}}dt =\int^{t_0}_0 \|K(T(t)x+\Phi_t g)\|^2_Hdt
\end{align}
where $(\Phi_t)_{t\ge 0}$ is the family of control maps associated with $A$ and $B$ (see \eqref{control-maps}). According to \cite[Lemma 3.6]{HaddManzoRhandi}, the operator $K$ coincides with $\mathscr{P}_\Lambda$ on $\mathscr{Z}$. By assumption $(A,B,P)$ is regular, so $T(t)x\in D(\mathscr{P}_\Lambda)$ and $\Phi_t g\in D(\mathscr{P}_\Lambda)$ for a.e. $t>0,$ and
\begin{align*}
\int^{t_0}_0 \|\mathscr{P}_\Lambda T(t)x\|^2_H\le c_1 \|x\|^2\quad\text{and}\quad \int^{t_0}_0 \|\mathscr{P}_\Lambda \Phi_t g\|^2_H\le c_2 \|g\|^2_{L^2([0,t_0],U)}
\end{align*}
for some constants $c_1:=c_1(t_0)>0$ and  $c_2:=c_2(t_0)>0$. Thus by \eqref{nnn}, $\calK$ is an admissible observation operator for $\mathfrak{A}$. Let $\calK_\Lambda$ be the Yosida extension of $\calK$ for $\mathfrak{A}$. In view of Theorem \ref{new variation}, the problem \eqref{Distr-Big} has a unique mild solution $\varrho(t)\in D(\calK_\Lambda)$ for a.e. $t\in (0,\infty)$, which satisfies
\begin{align}\label{mmm}
\varrho(t)=\mathfrak{T}(t)\left(\begin{smallmatrix}\xi\\u\end{smallmatrix}\right)+\int^t_0 \mathfrak{T}(t-s)\calK_{\Lambda}\varrho(s)ds+
\int^t_0 \mathfrak{T}(t-s)\mathscr{B}(\varrho(s))dW(s)
\end{align}
for any $t\ge 0$, $\xi\in L^2_{\calF_0}(\Om,H)$ and $u\in L^2_{\F}(0,+\infty;U)$. Denote by $\calC_{\F}^c(0,+\infty;L^2(\Om,U))$ the space of functions in $\calC_{\F}(0,+\infty;L^2(\Om,U))$ with compact support. First, we assume that $u\in \calC_{\F}^c(0,+\infty;L^2(\Om,U))$. From the first part of proof we have $\varrho(t)\in D(\mathscr{P}_\Lambda)\times \calC_{\F}^c(0,+\infty;L^2(\Om,U))$. Now for $\la>0$ sufficiently large, we have
\begin{align}\label{kkak}
\calK \la R(\la,\mathfrak{A})\varrho(t)=\begin{pmatrix} \mathscr{P}\la R(\la,A)X(t)+\la \calG(\la)\hat{u}(\la)\\ 0\end{pmatrix}
\end{align}
where $\calG(\la)=KR(\la,A_{-1})B$ is the transfer function of the system $(A,B,P)$ and $\hat{u}$ is the Laplace transform of $t\mapsto u(t)$ $\P$-a.s. We know that $\calG(\la)\to 0$ as $\la\to+\infty$. Moreover, using the fact that $u$ has a compact support, clearly  we have $\la \hat{u}(\la)\to 0$. By letting $\la \to+\infty$ in \eqref{kkak}, we obatin
\begin{align}\label{jjj}
\calK_\Lambda \varrho(t)=\begin{pmatrix} \mathscr{P}_{\Lambda}X(t)\\0\end{pmatrix},
\end{align}
for a.e. $t>0$ and $\P$-a.s. Thus, by combining  \eqref{mmm} and \eqref{kkak}, the formula \eqref{New-formula-B} holds for any $u\in \calC_{\F}^c(0,+\infty;L^2(\Om,U))$. Now for $u\in L^2_{\F}(0,+\infty;U)$, there exists a sequence $(u_n)_n\subset \calC_{\F}^c(0,+\infty;L^2(\Om,U))$ such that
\begin{align*}
\lim_{n\to+\infty}\|u_n-u\|_{L^2_{\F}(0,+\infty;U)}=0.
\end{align*}
Let $X(\cdot,\xi,u)$ and $X^n(\cdot,\xi,u_n)$ be the solutions of \eqref{P-ABC} corresponding to the input functions $u$ and $u_n$, respectively. By using the equation \eqref{Distr-Big} and the fact that $(\mathfrak{A}+\calK,D(\mathfrak{A}))$ is a generator of a strongly continuous semigroup on $\mathfrak{X}$, one can see that for any $\al>0,$ and $t\ge 0,$
\begin{align*}
\E \|\varrho_n(t,\xi,u_n)-\varrho(t,\xi,u)\|^2\le c(\al) \|u_n-u\|_{L^2_{\F}(0,\al;U)}.
\end{align*}
This implies that $\E\|X^n(t)-X(t)\|^2$ geos to zero at infinity. Observe that $X^n(\cdot,\xi,u_n)$ satisfies the equation \eqref{New-formula-B}. Now by selecting
\begin{align*}
\X(t)=T(t)\xi+\int^t_0 T(t-s)\mathscr{P}_{\Lambda}X(s)ds&+\int^t_0 T_{-1}(t-s)Bu(s)ds\cr &+\int^t_0 T(t-s)\mathscr{M}(X(s))dW(s)
\end{align*}
and using a similar argument as in Theorem \ref{new variation}, we obtain $\E\|X^n(t)-\X(t)\|^2\to 0$ as $t\to +\infty$. Thus $X(t)=\X(t)$ for any $t\in [0,\al]$ and $\P$-a.s. This ends the proof.
\end{proof}

\section{Application to stochastic systems with state, input and output delays}\label{sec:delay}
Consider the stochastic delay system
\begin{align}\label{SSdelay}
\begin{cases}
dX(t)=\left(A+\displaystyle\int^0_{-r}d\mu(\theta)X(t+\theta)+\int^0_{-r}d\eta(\theta)u(t+\theta)\right)dt
\cr \hspace{7cm}+\Upsilon(X(t))dW(t),& t\ge 0,\cr X(0)=\xi,\quad X(\theta)=\varphi,\quad u(\theta)=\psi(\theta),& \theta\in [-r,0],
\cr Y(t)= \displaystyle\int^0_{-r}d\vartheta(\theta)X(t+\theta)+\int^0_{-r}d\varpi(\theta)u(t+\theta),
\end{cases}
\end{align}
where $A:D(A)\subset H\to H$ is a generator of a strongly continuous semigroup $\T=(T(t))_{t\ge 0}$ on $H,$ the operator $\Upsilon\in\calL(H),$
the family $W(t)$ as in the previous sections, $r>0$ is a real number, $\mu:[-r,0]\to \calL(H),$  $\eta:[-r,0]\to \calL(U,H),$ $\vartheta:[-r,0]\to \calL(H,U),$
and $\varpi:[-r,0]\to \calL(U),$
are functions of bounded variations with total variations $|\mu|,$ $|\eta|,$ $|\vartheta|,$ and $|\varpi|$, respectively. The initial conditions
$\xi\in L^2_{\calF_0}(\Om,H)$, $\varphi\in L^2_{\calF_0}(\Om,L^2([-r,0],H))$, and $\varphi\in L^2_{\calF_0}(\Om,L^2([-r,0],U))$ and
the segment processes $X_t:[-r,0]\to H$ and $u_t:[-r,0]\to U$ ($t\ge 0$) are defined by
\begin{align*}
X_t(\theta)=\begin{cases}X(t+\theta),& -t\le \theta\le 0,\cr \varphi(t+\theta),& -r\le \theta<-t,\end{cases},\quad
u_t(\theta)=\begin{cases}u(t+\theta),& -t\le \theta\le 0,\cr \psi(t+\theta),& -r\le \theta<-t.\end{cases}
\end{align*}

In order to reformulate the stochastic delay system  \eqref{SSdelay} as a free-delay system, we introduce the product spaces
\begin{align*}
\calH&:=H\times L^2([-r,0],H),\qquad \left\|\left(\begin{smallmatrix} x\\\phi_1\end{smallmatrix}\right)\right\|_{\calH}:= \|x\|_H+\|\phi_1\|_{L^2([-r,0],H)},\cr
\calX&:=\calH\times\times L^2([-r,0],U),\quad \left\|\left(\begin{smallmatrix} x\\\phi_1\\ \phi_2\end{smallmatrix}\right)\right\|
:=\left\|\left(\begin{smallmatrix} x\\\phi_1\end{smallmatrix}\right)\right\|_{\calH}+\|\phi_2\|_{L^2([-r,0],U)}.
\end{align*}
 Consider the following Riemann Stieltjes integrals
\begin{align*}
& R^\mu \phi_1=\int^0_{-r}d\mu(\theta)\phi_1(t+\theta),\quad R^\vartheta \phi_1=\int^0_{-r}d\vartheta(\theta)\phi_1(t+\theta),\quad \phi_1\in W^{1,2}([-r,0],H),\cr
& R^\eta \phi_2=\int^0_{-r}d\mu(\theta)\phi_2(t+\theta),\quad R^\varpi\phi_2=\int^0_{-r}d\varpi(\theta)\phi_2(t+\theta),\quad \phi_2\in W^{1,2}([-r,0],U).
\end{align*}
It is known (see e.g. \cite{Hadd-SF}) that the following operator
\begin{align*}
\calA:=\begin{pmatrix} A&R^\mu\\ 0& \frac{d}{d\si}\end{pmatrix}, \;
D(\calA)=\left\{ \left(\begin{smallmatrix} x\\\phi_1\end{smallmatrix}\right)
\in D(A)\times W^{1,2}([-r,0],H):\phi_1(0)=x\right\}
\end{align*}
generates a $C_0$-semigroup $\calT:=(\calT(t))_{t\ge 0}$ on $\calH$.
We also need the following operators
\begin{align*}
 \calA_m:=\left(
\begin{array}{c|c}
 \calA &\begin{matrix}
0\\0
\end{matrix}\\
\hline
\begin{matrix}
0&0
\end{matrix}& \frac{d}{d\theta}
\end{array}\right),\quad \mathscr{K}:=\left(
\begin{array}{c|c}
 0_{\calH}&\begin{matrix}
R^\eta\\0
\end{matrix}\\
\hline
\begin{matrix}
0&0
\end{matrix}& 0
\end{array}\right)
\end{align*}
with domain $D(\calA_m):=D(\mathfrak{A}_\mu)\times W^{1,2}([-r,0],U)$. In addition, we denote
\begin{align*}
&\mathscr{G}\left(\begin{smallmatrix} x\\\phi_1\\\phi_2\end{smallmatrix}\right):=\phi_2(0), \quad
\calM\left(\begin{smallmatrix} x\\\phi_1\\\phi_2\end{smallmatrix}\right):=R^\vartheta\phi_1+R^\varpi\phi_2,\qquad \left(\begin{smallmatrix} x\\\phi_1\\\phi_2\end{smallmatrix}\right)\in D(\calA_m)\cr &
\mathscr{M}\left(\begin{smallmatrix} x\\\phi_1\\\phi_2\end{smallmatrix}\right)=
\left(\begin{smallmatrix} \Upsilon(x)\\0\\ 0\end{smallmatrix}\right).
\end{align*}
By introducing the new state
\begin{align*}
Z(t):=\left(\begin{smallmatrix} X(t)\\ X_t\\ u_t\end{smallmatrix}\right),\quad t> 0,\quad Z(0):=Z^0= \left(\begin{smallmatrix} \xi\\\varphi\\ \psi\end{smallmatrix}\right),
\end{align*}the delay system \eqref{SSdelay} is transformed to the following free-delay system
\begin{align}\label{free-delay}
\begin{cases}
dZ(t)=\left(\calA_m+\mathscr{K}\right) Z(t) dt + \mathscr{M}(Z(t))dW(t), & t> 0, \quad Z(0)=Z^0,\cr
 \mathscr{G} Z(t)=u(t),& t\ge 0,\cr Y(t)=\calM Z(t),& t\ge 0.
\end{cases}
\end{align}
This reformulation allows us to introduce the following definition.
\begin{definition}\label{delay-well-posed-def}
The stochastic system \eqref{SSdelay} is well-posed if the system \eqref{free-delay} is so.
\end{definition}
To state the main result of this section, we denote by $\calQ^H$ and $\calQ^U,$ the generators of the left shift semigroups $\calS^H$ and $\calS^H$ on $L^2([-r,0],H)$ and $L^2([-r,0],U)$  defined by
\begin{align*}
\calS^H (t)\phi_1=\1_{\{t+\cdot\le 0\}}\phi_1(t+\cdot),\quad \calS^U (t)\phi_2=\1_{\{t+\cdot\le 0\}}\phi_2(t+\cdot)
\end{align*}
for $t\ge 0,$  $\phi_1\in L^2([-r,0],H)$ and $\phi_2\in L^2([-r,0],U)$.
\begin{theorem}\label{thm-Well-posed-delay}
The stochastic delay system \eqref{SSdelay} is well-posed. Moreover,
for $\xi\in L^2_{\calF_0}(\Om,H)$,
$\varphi\in L^2_{\calF_0}(\Om,L^2([-r,0],H))$, $\psi\in L^2_{\calF_0}(\Om,L^2([-r,0],U))$,
and any $\al>0$, there is a constant $c:=c(\al)>0$ such that for any $u\in L^2_{\F}(0,\al;U)$
\begin{align*}
&\| R^\vartheta_{\Lambda}X_{\bullet}+R^\varpi_\Lambda u_{\bullet}\|_{L^2_{\F}(0,\al;U)}\cr & \quad\le
c \left(\|\xi\|_{L^2_{\calF_{0}}(\Om,H)}+\|\varphi\|_{L^2_{\calF_{0}}(\Om,L^2([-r,0],H))}+
\|\psi\|_{L^2_{\calF_{0}}(\Om,L^2([-r,0],U))}+\|u\|_{L^2_{\F}(0,\al;U)}\right),
\end{align*}
where $ R^\vartheta_{\Lambda}$ and $R^\varpi_\Lambda$ are the Yosida extensions of $R^\vartheta$ and $R^\varpi$ for $\calQ^H$ and $\calQ^U,$ respectively.
\end{theorem}
\begin{proof}
It suffices to prove that the system \eqref{free-delay} is well-posed. To this end, we will apply Theorem \ref{bounadry-theorem}. In fact, the diagonal  operator
\begin{align*}
\mathfrak{A}:=\calA_m,\quad D(\mathfrak{A})=D(\calA)\times D(\calQ^U),
\end{align*}
generates a strongly continuous semigroup $\mathfrak{T}:=(\mathfrak{T}(t))_{t\ge 0}$ on $\calX$. On the other hand, the
operator $\mathscr{G}$ is surjective. As for $\la\in\rho(\mathfrak{A})=\rho(\calA),$ $\ker(\la-\calA_m)=\{0_{\calH}\}\times \{e_\la v:=e^{\la\cdot}v:v\in U\},$ one can easily see that the control operator associated with the system \eqref{free-delay} is
\begin{align*}
\mathfrak{B}=\left(\begin{smallmatrix}0\\0\\ b^U\end{smallmatrix}\right),
\end{align*}
where $b^U=(\la-\calQ^U_{-1}e_\la)$. As $b^U$ is an admissible control operator for $\calQ^U$ (see \cite{HA-Id-Rha-MCSS}), then $\mathfrak{B}$ is an admissible control operator for $\mathfrak{A}$. By using \cite[Lemma 6.2]{Hadd-SF}, one can see that $\mathscr{P}:=\mathscr{K}_{|D(\mathfrak{A})}$ and $\mathfrak{C}:=\cal{M}_{|D(\mathfrak{A})}$ are admissible observation operators for $\mathfrak{A}$. On the other hand,
according to \cite[Theorem 3]{HA-Id-Rha-MCSS}, the systems $(\calQ^U,b^U,R^\eta)$ and $(\calQ^H,b^H,R^\eta)$ are regular. Now  a similar argument as in \cite{Ha-Id-IMA}, \cite{Hadd-Zhong} shows that the systems $(\mathfrak{A},\mathfrak{B},\mathscr{P})$ and $(\mathfrak{A},\mathfrak{B},\mathfrak{C})$ are regular. Thus by Theorem \ref{bounadry-theorem}, the stochastic
system \eqref{free-delay} is well-posed and
\begin{align}\label{last-inega}
\| \mathfrak{C}_\Lambda Z(\cdot)\|_{L^2_{\F}(0,\al;U)}\le c
 \left( \left\|\left(\begin{smallmatrix} \xi\\\varphi\\ \psi\end{smallmatrix}\right)\right\|_{L^2_{\calF_0}(\Om,\calX)}+\|u\|_{L^2_{\F}(0,\al;U)}\right)
\end{align}
for $\al>0$ and a constant $c:=c(\al)>0$, where $\mathfrak{C}_\Lambda$ is the Yosida extension of $\mathfrak{C}$ for $\mathfrak{A}$. By \cite{Ha-Id-IMA},
we have
\begin{align*}
&\Sigma:=H\times [D(R^\mu_\Lambda)\cap D(R^\vartheta_\Lambda)]
\cap [D(R^\eta_\Lambda)\cap D(R^\varpi_\Lambda)]\subset D(\mathfrak{C}_\Lambda),\cr &
\mathfrak{C}_\Lambda=\begin{pmatrix}0& R^\vartheta_\Lambda & R^\varpi_\Lambda\end{pmatrix}\quad\text{on}\quad \Sigma.
\end{align*}
Using standard arguments as in \cite{Ha-Id-IMA}, we have $Z(t)\in \Sigma$ for a.e. $t\ge 0$ and $\P$-a.s. Thus the rest of the proof follows from
\eqref{last-inega}.
\end{proof}

\section{Observability of the stochastic semilinear system with unbounded observation operators}\label{sec:semilinear}
In this section, we introduce the concepts of admissibility and observability for semilinear stochastic systems. In fact we consider the system \eqref{semilinear-S}, where $C\in\calL(D(A),\mathscr{Y})$ and applications $G,F:H\to H$ satisfying
\begin{align}\label{Lip-f}
\begin{split}
 &\|G(x)-G(y)\|+\|F(x)-F(y)\|\leq L \|x-y\|,\quad\forall x,y\in H,\\
&\|G(x)\|^{2}+\|F(x)\|^{2}\leq \beta (1+\|x\|^{2}),\quad\forall x\in H.
\end{split}
\end{align}
for  constants $L>0$ and $\beta>0$.
It is well-known, see e.g. \cite{Da-Za}, that the condition \eqref{Lip-f} implies the existence of a unique mild solution $X(\cdot)\in \calC_{\F}(0,+\infty;L^2(\Om,H))$ of the system \eqref{semilinear-S} satisfying
\begin{align}\label{mild-sol}
X(t)=T(t)\xi+\int^t_0 T(t-s)G(X(s))ds+\int^t_0 T(t-s)F(X(s))dW(s)
\end{align}
for any $t\ge 0$ and $\xi\in L^2_{\calF_0}(\Om,H)$.
\begin{definition}\label{CAW-wellposed}
Let the condition \eqref{Lip-f} be satisfied. We say that the semilinear stochastic system \eqref{semilinear-S}  is {\em well-posed} if there exists an extension $\tilde{C}:D(\tilde{C})\subset H\to \mathscr{Y}$ of the operator $C$ such that $X(t;\xi)\in D(\tilde{C})$ for a.e. $t>0$ and $\mathbb{P}$-a.s and all $\xi\in L^2_{\calF_0}(\Om,H)$, and for any $\al>0,$ there exists a constant $c:=c_\al$ such that for all $\xi_1,\xi_2\in L^2_{\calF_0}(\Om,H)$,
\begin{align*}
\E\int^\al_0 \|\tilde{C}\left(X(t;\xi_1)-X(t;\xi_2)\right)\|^2_{\mathscr{Y}}dt\le c^2 \E\|\xi_1-\xi_2\|^2.
\end{align*}
\end{definition}
The following result gives the well-posedness of the system \eqref{semilinear-S}.
\begin{theorem}\label{output-rep}
Let the condition \eqref{Lip-f} be satisfied and assume that $C$ is an admissible observation operator for $A$. The following assertions hold.

\begin{itemize}
  \item [{\rm (i)}] The mild solution of the system \eqref{semilinear-S} satisfies $X(t,\xi)\in D(C_\Lambda)$ for a.e. $t>0$, $\P$-a.s and all $\xi\in L^2_{\calF_0}(\Om,H)$.
  \item [{\rm (ii)}]For any $\al>0,$ there exists a constant $\delta:=\delta(\al)>0$ such that for any $\xi_1,\xi_2\in L^2_{\calF_0}(\Om,H),$ we have
\begin{align*}
  \|C_\Lambda X(\cdot,\xi_{1})-C_\Lambda X(\cdot,\xi_{2})\|_{L^2_{\F}(0,\t;\mathscr{Y})}\leq \delta\|\xi_{1}-\xi_{2}\|_{L^2_{\calF_0}(\Om,H)},
  \end{align*}
\end{itemize}
In particular, the system \eqref{semilinear-S} is well-posed.
\end{theorem}
\begin{proof}
The  assertion (i) follows by combining  Proposition \ref{Hadd-SF-prop} and Proposition \ref{fondamental-lemma}. Let us now prove the assertion (ii). Let $\xi_1,\xi_2\in L^2_{\calF_0}(\Om,H)$ and $\al>0$. By using \eqref{mild-sol} and Gronwall's inequality, we have
 \begin{align}\label{E-estim}
 \E\|X(t,\xi_{1})-X(t,\xi_{2})\|^{2}_H\le \upsilon_\al \E\|\xi_1-\xi_2\|^2_H,\qquad t\in [0,\al].
 \end{align}
Let $\Psi$ be the extended output map associated with the admissibility of $C$ for $A$. By using Proposition \ref{Hadd-SF-prop}, Proposition \ref{fondamental-lemma} and the condition \eqref{Lip-f}, we obtain
 \begin{align*}
 & \E\int^\al_0\|C_\Lambda (X(t,\xi_{1})-X(t,\xi_{2}))\|^2_{\mathscr{Y}} dt\le  3\E\int^\al_0\|(\Psi(\xi_1-\xi_2))(t)\|^2_{\mathscr{Y}}dt\cr  &\qquad  +
 3 \E\int^\al_0 \left\|C_\Lambda \int^t_0 T(t-s)(G(X(s,\xi_1))-G(X(s,\xi_2))) ds\right\|^2_{\mathscr{Y}}dt\cr & \qquad\qquad+ 3 \E\int^\al_0 \left\|C_\Lambda \int^t_0 T(t-s)(F(X(s,\xi_1))-F(X(s,\xi_2))) dW(s)\right\|^2_{\mathscr{Y}}dt\cr & \qquad\le 3\ga^2 \mathbb{E}\|\xi_{1}-\xi_{2}\|^{2}_{H}+  3 \ga^2 \al \E\int^\al_0 \|G(X(s,\xi_1))-G(X(s,\xi_2))\|^2_Hds\cr & \qquad\qquad +3\ga^2 \E\int^\al_0 \|F(X(s,\xi_1))-F(X(s,\xi_2))\|^2_Hds \cr & \qquad \le \delta^2 \mathbb{E}\|\xi_{1}-\xi_{2}\|^{2}_H
 \end{align*}
 for a constant $\delta:= \ga \sqrt{3 (1+ \nu_\al L^2 (\al^2+\al))}$, where $L$ is the Lipschitz condition given in \eqref{Lip-f}, and $\ga$ is the admissibility constant associated with the operators $C$ and $A$.
\end{proof}
The following concept of observability is well known for deterministic linear systems.
\begin{definition}\label{obs-ds}
Let $C\in\calL(D(A),Y)$ be an admissible observation operator for $A$ and let $\Psi$ be the associated extended output map. For $\t>0$, we say that $(C,A)$ is $\t$-exactly observable if there exists a constant $\ga_\t>0$ such that for any $x\in X,$ we have
\begin{align*}
\int^\t_0 \|(\Psi x)(t)\|^2_{\mathscr{Y}}dt\ge \ga_\t^2 \|x\|^2_H.
\end{align*}
\end{definition}
By analogy to the exact observability of linear systems, we introduce the following definition for semilinear stochastic systems.
\begin{definition}\label{Def-obs-nl}
Let the condition \eqref{Lip-f} be satisfied and assume that $C$ is an admissible observation operator for $A$. For some $\tau>0,$ we say that the system \eqref{semilinear-S} is $\tau$-exactly observable if there exists a constant $\kappa_{\tau}>0$ such that for any $\xi_{1},\xi_{2}\in L^2_{\calF_0}(\Om,H),$
\begin{equation*}
\|C_\Lambda X(\cdot;\xi_{1})-C_\Lambda X(\cdot;\xi_{2})\|_{L^{2}_{\F}(0,\tau;\mathscr{Y})}\geq \kappa_{\tau} \|\xi_{1}-\xi_{2}\|_{L^2_{\calF_0}(\Om,H)}.
\end{equation*}
\end{definition}
Now we can state the main result of this section.
\begin{theorem} \label{Stochastic-oservability-result}
Let the condition \eqref{Lip-f} be satisfied and assume that the pair $(C,A)$ is $\tau$-exactly observable for some $\tau>0$.
Then there exists a constant $\Theta_{\t}>0$ such that the stochastic semilinear system \eqref{semilinear-S} is $\t$--exactly observable whenever the Lipschtiz constant $L<\Theta_{\t}$.
\end{theorem}
\begin{proof}
Let $\xi_1,\xi_2\in L^2_{\calF_0}(\Om,H)$. By assumption, there exists $\delta_\t>0$ such that
\begin{align*}
\E \int^\t_0 \|(\Psi(\xi_1-\xi_2))(t)\|^2_{\mathscr{Y}}dt\ge \delta_\t^2 \E\|\xi_1-\xi_2\|^2_H.
\end{align*}
On the other hand, we write the mild solution \eqref{mild-sol} as
\begin{align*}
X(t;\xi)=T(t)\xi+\K_t (X(\cdot;\xi))
\end{align*}
with
\begin{align*}
\K_t (X(\cdot;\xi)):=\int_{0}^{t}T(t-s)G(X(s))ds+\int_{0}^{t}T(t-s)F(X(s))dW(s).
\end{align*}
Let $\xi_1,\xi_2\in L^2_{\calF_0}(\Om,H)$. As in the proof of Theorem \ref{output-rep}, Proposition \ref{Hadd-SF-prop}, Proposition \ref{fondamental-lemma} and the estimate \eqref{E-estim}, we obtain
\begin{align*}
\E \int^\t_0 \|C_\Lambda \left(\K_t (X(\cdot;\xi_1))-\K_t (X(\cdot;\xi_2))\right)\|^2_{\mathscr{Y}}dt\le a_\t^2 L^2 \E\|\xi_1-\xi_2\|^2_H,
\end{align*}
for a constant $a_\t^2:=2 \ga^2 \upsilon_\t (\t^2+\t)$. Moreover, we have
\begin{align*}
  \mathbb{E}\int_{0}^{\t}\|C_\Lambda X(t,\xi_{1})-C_\Lambda X(t,\xi_{2})\|^{2}_{\mathscr{Y}} dt &\ge \E \int^\t_0 \|C_\Lambda T(t)(\xi_1-\xi_2)\|^2_{\mathscr{Y}}dt\cr &  -\E \int^\t_0 \|C_\Lambda \left(\K_t (X(\cdot;\xi_1))-\K_t (X(\cdot;\xi_2))\right)\|^2_{\mathscr{Y}}dt\cr &  \ge \left(\delta_\t^2-a_\t^2L^2\right)\E\|\xi_1-\xi_2\|^2_H\cr & :=h(L) \E\|\xi_1-\xi_2\|^2_H,
  \end{align*}
where $h(L)=\delta_\t^2-a_\t^2 L^2$. The function $h$ is continuous from $\R^+$ to $(-\infty,\delta_\t^2]$ and strictly decreasing, hence it is bijective. Then there exists a unique  $\Theta_\t>0$ such that $h(\Theta_\t)=0$. Hence $h(L)>0$ whenever $L<\Theta_\t$. Thus the system \eqref{semilinear-S} is $\t$-observable whenever $L<\Theta_\t$.
\end{proof}
\begin{example}\label{Example1}
Let $\mathscr{O}=(0,\pi)\times (0,\pi)$,  $\Gamma$ be a open nonempty set of $\partial \mathscr{O}$, $q(x)$ some polynomial with impair degree, and $\eta_1,\eta_2\in (0,+\infty)$. We consider the semi-linear Dirichlet boundary observation Schr\"odinger system
\begin{align}\label{schro-exa}
\begin{cases}
dX(t,x)=\left(-i\Delta X(t,x)+ \eta_1 \left(X(t,x)\right)^3\right)dt\cr \hspace{5.5cm}+\eta_2 q(X(t,x)) dW(t,x),& x\in \mathscr{O},\;t\ge 0,\cr X(0,x)=\xi(x),& x\in \mathscr{O},\cr \displaystyle\frac{\partial}{\partial\nu} X(t,x)=0,& x\in \partial\mathscr{O},\;t\ge 0,\cr Y(t,x)=X(t,x)& x\in \Gamma,\; t\ge 0.
\end{cases}
\end{align}
Define the state space $H=L^2(\mathscr{O})$ and the control space $U=L^2(\Gamma)$. Moreover, we consider the following operators
\begin{align*}
& A\psi= -i\Delta \psi,\qquad D(A)=\left\{\psi\in H^2(\mathscr{O}): \frac{\partial}{\partial\nu} \psi=0 \right\}\cr & C\psi=\psi_{|\Gamma},\qquad \psi\in D(A)\cr & G(\psi)(x)= \eta_1(\psi(x))^3,\quad F(\psi)(x)=\eta_2 q(\psi(x)),\qquad x\in \mathscr{O},\quad \psi\in H.
\end{align*}
Thus the system \eqref{schro-exa} takes the form of the abstract form \eqref{semilinear-S}. On the other hand, it is well-known that the operator $A$ is   skew-adjoint  and generates a strongly continuous semigroup $\T:=(T(t))_{t\ge 0}$ on $H$. Moreover, according to \cite[Proposition 3.2]{RTTT-JFA}, there exists $\t>0$ such that $(C,A)$ is $\t$-exactly observable. In addition,  the applications $G,F:H\to H$ satisfy the condition \eqref{Lip-f}. Now according to Theorem \ref{Stochastic-oservability-result}, the system  \eqref{schro-exa} is $\t$-exactly observable whenever $\eta=\max(\eta_1,\eta_2)$ is small enough.
\end{example}
\section*{Acknowledgments} We would like to thank the editors and the referees whose detailed comments helped us to improve the organization and the content of the paper.


\begin{thebibliography}{}
\bibitem{Alb} S. Albeverio, F. C. De Vecchi1,  A. Romano and S. Ugolini,
 {\em Mean--field limit for a class of stochastic ergodic control problems,}
 https://arxiv.org/abs/2003.06469

\bibitem{Alb1} S. Albeverio, L.D. Persio and E. Mastrogiacomo,
{\em Invariant measures for stochastic differential equations on networks,}
1–33, Proc. Sympos. Pure Math., 87, Amer. Math. Soc., RI, 2013.

\bibitem{Alb2}S. Albeverio, L.D. Persio, E. Mastrogiacomo and B. Smii,
{\em A class of L\'{e}vy driven SDEs and their explicit invariant measures,}
Potential Anal., 45 (2016), pp. 229--259.

\bibitem{Alb3}S. Albeverio, L.D. Persio, E. Mastrogiacomo and B. Smii,
{\em Invariant measures for SDEs driven by L\'{e}vy noise: A case study for dissipative nonlinear drift in infinite dimension,}
Commun. Math. Sci. 15 (4) (2017), pp. 957--983.

\bibitem{Al-Bo-02} E. Al\`{o}s and S. Bonaccorsi,
{\em Stochastic partial differential equations with Dirichlet white-noise boundary conditions,}
Ann. Inst. H. Poincaré Probab. Statist., 38 (2002), pp. 125--154.

\bibitem{ABDH-20} A. Amansag, H. Bounit, A. Driouich and S. Hadd,
{\em On the maximal regularity for perturbed autonomous and nonautonomous evolution equations,}
J. Evol. Equ., 20 (2020), pp. 165--190.

\bibitem{BJ-09} M. Baroun and B. Jacob,
{\em Admissibility and Observability of Observation Operators for Semilinear Problems,}
Integr. Equ. Oper. Theory, 64 (2009), pp. 1--20.

\bibitem{cerrai0} S. Cerrai,
{\em Optimal control problems for stochastic reaction-diffusion systems with non-Lipschitz coefficients,}
 SIAM J. Control Optimi., 39 (2001), pp. 1779--1816.

\bibitem{cerrai1} S. Cerrai,
{\em Stationary Hamilton-Jacobi equations in Hilbert spaces and applications to a stochastic optimal control problem,}
 SIAM J. Control Optimi., 40 (2002), pp. 824--852.

\bibitem{Cu-Zw}   R.F. Curtain and H. Zwart,
{\em Introduction to Infinite--Dimensional Linear Systems,}
TMA 21, Springer--Verlag, New York, 1995.

\bibitem{Da-Za-SSR} G. Da Prato and J. Zabczyk,
{\em Evolution equations with white-noise boundary conditions,}
Stochastics and Sochastics Rep., 42 (1993), 167--182.

\bibitem{Da-Za} G. Da Prato and J. Zabczyk,
\emph{Stochastic Equations in infinite Dimensions,}
Cambridge University Press, 2014.

\bibitem{Da-Za-Book-96} G. Da Prato and J. Zabczyk,
{\em Ergodicity for infinite dimensional systems,}
London Math. Soc. Lect. Not. Ser. 229, Cambridge Univ. Press, 1996.

\bibitem{Duncan-12} T.E. Duncan, B. Maslowski and B. Pasik-Ducan,
{\em Linear-Quadratic control for stochastic equations in a Hilbert space with fractional Brownian motions,}
SIAM J. Control Optim., 50 (2012), pp. 507--531.

\bibitem{Duncan-13} T.E. Duncan and B. Pasik-Ducan,
{\em Linear-Quadratic fractional Gaussian control,}
SIAM J. Control Optim., 51 (2013), pp. 4504--4519.

\bibitem{EngNag} K.-J. Engel,  R. Nagel,
{\em One-Parameter Semigroups for Linear Evolution Equations,}
Springer-Verlag, New York, Berlin, Heidelberg, 2000.

\bibitem{Fa-GO-09}G. Fabbri and B. Goldys,
{\em An LQ problem for the heat equation on the halfline with Dirichlet boundary control and noise,}
 SIAM J. Control Optim., 48 (2009), pp. 1473--1488.

\bibitem{flandoli1} F. Flandoli,
{\em Dirichlet Boundary value problem for stochastic parabolic equations: compatibility relations and regularity of solutions,}
Stochastics and Stochastic Reports, 29 (1990), pp. 331--357.

\bibitem{Grei} G. Greiner,
{\em Perturbing the boundary conditions of a generator,}
Houston J. Math., 18 (2001), pp. 405--425.

\bibitem{govi-appr-book} T. E. Govindan,
{\em Yosida Approximations of Stochastic Differential Equations in Infinite Dimensions and Applications,}
Springer, Switzerland, 2016.

\bibitem{Guo-sha} B.Z. Guo and Z.C. Shao,
{\em Regularity of a Schr\"odinger equation with Dirichlet control and colocated observation,}
Systems Control Letters, 54 (2005), pp. 1135--1142.

\bibitem{Hadd-SF} S. Hadd,
{\em Unbounded perturbations of $C_0$-semigroups on Banach spaces and applications,}
Semigroup Forum, 70 (2005), pp. 451--465.

\bibitem{Ha-Id-IMA} S. Hadd, A. Idrissi,
{\em Regular linear systems governed by systems with state, input and output delays,}
 IMA J. of Math. Control Inform., 22 (2005), pp. 423--439.

\bibitem{Ha-Id-SCL} S. Hadd, A. Idrissi,
{\em On the admissibility of observation for perturbed $C_0$--semigroups on Banach spaces,}
Systems Control Letters, 55 (2006), pp. 1--7.

\bibitem{HA-Id-Rha-MCSS} S. Hadd, A. Idrissi and A. Rhandi,
{\em The regular linear systems associated with the shift semigroups and application to control linear systems with delay,}
 Math. Control Signals Sys., 18 (2006), pp. 72--291.

\bibitem{HaddManzoRhandi} S. Hadd, R. Manzo, A. Rhandi,
{\em Unbounded perturbations of the generator domain,}
Discrete Continuous Dyn. Sys. A, 35 (2015), pp. 703--723.

\bibitem{Hadd-Zhong} S. Hadd and Q.C. Zhong,
{\em On feedback stabilizability of linear systems with state and input delays in Banach spaces,}
IEEE Trans. Automatic control 54 (2009), pp. 438--451.

\bibitem{MP-10} Z.-D. Mei and J.-G. Peng,
{\em On the perturbations of regular linear systems and linear systems with State and output delays},
Integr. Equ. Oper. Theory, 68 (2010), pp. 357--381

\bibitem{LW-19} F. Lamoline and J. Winkin,
{\em Well-Posedness of Boundary Controlled and Observed Stochastic Port-Hamiltonian Systems,}
 IEEE Trans. Aut. Control, 65 (2019), pp. 4258--4264

\bibitem{LT} I. Lasiecka and R. Triggiani,
{\em Control theory for partial differential equations: continuous and approximation theories. II. Abstract hyperbolic--like systems over a finite time horizon,}
Encyclopedia of Mathematics and its Applications, 75 Cambridge University Press, Cambridge, England, 2000.

\bibitem{LeMery} C. Le Merdy,
{\em The Weiss conjecture for bounded analytic semigroups,}
J. London Math. Soc. 67 (2003), pp. 715--738.

\bibitem{Lu-JDE-13} Q. L\"u,
\emph{Exact controllability for stochastic Schrödinger equations,}
J. Diff. Equa., 255 (2013), pp. 2484--2504.

\bibitem{Lu-SIAM-13} Q. L\"u,
\emph{Observability estimate for stochastic Schrodinger equations and its applications,}
SIAM J. Control  Optim., 51 (2013), pp. 121--144.

\bibitem{Lu-SIAM-15} Q. L\"u,
\emph{Stochastic Well-posed Systems and Well-posedness of Some Stochastic Partial Differential Equations with Boundary Control and Observation,}
 SIAM J. Control Optimization,  53 (2015),  pp. 3457--3482.

\bibitem{Lu-Zhang-21} Q. L\"u and X. Zhang, \emph{A Concise Introduction to Control Theory for Stochastic Partial
Differential Equations,} arXiv:2101.10678

\bibitem{Nazim2001} N. I. Mahmudov,
\emph{Controllability of linear stochastic systems in Hilbert spaces,}
J. Math. Anal. Appl., 259 (2001), pp. 64--82.

\bibitem{RTTT-JFA} K. Ramdani, T. Takahashi, G. Tenenbaum and M. Tucsnak,
{\em A spectral approach for the exact observability of infinite-dimensional systems with skew-adjoint generator,}
J. Funct. Anal., 226 (2005) pp. 193--229.

\bibitem{Sala} D. Salamon,
{\em Infinite-dimensional linear system with unbounded control and observation: a functional analytic approach,}
Trans. Amer. Math. Soc. 300 (1987) pp. 383--431.

\bibitem{Staf} O.J. Staffans,
{\em Well-Posed Linear Systems,}
Cambridge Univ. Press, Cambridge, 2005.

\bibitem{TucWei} M. Tucsnak and  G. Weiss,
{\em Observation and control for operator semigroups,}
Birkhäuser, 2009.

\bibitem{WeiRegu} G. Weiss,
{\em Regular linear systems with feedback,}
Math. Control Signals Sys., 7 (1994), pp. 23--57.

\bibitem{WeiTrans} G. Weiss,
{\em Transfer functions of regular linear systems. Part I: Characterization of regularity,}
 Trans. Amer. Math. Soc., 342 (1994), pp. 827--854.


\end{thebibliography}
\end{document}